\documentclass[11pt,reqno]{amsart}
\usepackage{eucal,fullpage,times,amsmath,amsthm,amssymb,mathrsfs,stmaryrd,enumerate,accents}
\usepackage[all]{xy}
\usepackage{url}

\usepackage[usenames]{color}
\usepackage{aliascnt} 
\usepackage[colorlinks=true,linkcolor=blue,citecolor=blue,urlcolor=blue,citebordercolor={0 0 1},urlbordercolor={0 0 1},linkbordercolor={0 0 1},hyperindex]{hyperref} 

\newcommand{\calO}{{\mathcal{O}}}
\newcommand{\calA}{{\mathcal{A}}}
\newcommand{\calB}{{\mathcal{B}}}
\newcommand{\calC}{\mathcal{C}}

\newcommand{\Z}{\mathbf{Z}}
\newcommand{\G}{\mathbf{G}}

\newcommand{\A}{\mathbf{A}}

\renewcommand{\P}{\mathbf{P}}

\newcommand{\Spec}{{\mathrm{Spec}}}

\newcommand{\QAff}{\mathrm{QAff}}

\newcommand{\CAlg}{\mathrm{CAlg}}

\newcommand{\Hom}{\mathrm{Hom}}
\newcommand{\Map}{\mathrm{Map}}
\newcommand{\Sym}{\mathrm{Sym}}
\newcommand{\Fun}{\mathrm{Fun}}
\newcommand{\Func}{\mathrm{Fun}^c}

\newcommand{\Tor}{\mathrm{Tor}}
\newcommand{\QCoh}{\mathrm{QCoh}}
\newcommand{\Ind}{\mathrm{Ind}}
\newcommand{\perf}{\mathrm{perf}}
\newcommand{\Coh}{\mathrm{Coh}}

\newcommand{\id}{\mathrm{id}}
\newcommand{\coker}{\mathrm{coker}}

\renewcommand{\ker}{\mathrm{ker}}

\newcommand{\opp}{\mathrm{opp}}
\newcommand{\Op}{\mathrm{Op}}
\newcommand{\Sp}{\mathcal{S}}

\newcommand{\colim}{\mathop{\mathrm{colim}}}

\newcommand{\adjunction}[4]{\xymatrix@1{#1{\ } \ar@<0.3ex>[r]^{ {\scriptstyle #2}} & {\ } #3 \ar@<0.3ex>[l]^{ {\scriptstyle #4}}}}

\newcommand{\op}[1]{\!\!\mathop{\rm ~#1}\nolimits}
\newcommand{\Perf}{\op{Perf}}
\newcommand{\Bl}{\op{Bl}}

\newcommand{\APerf}{\op{APerf}}
\newcommand{\sod}[1]{\left\langle #1 \right\rangle}
\newcommand{\GL}{\op{GL}}
\newcommand{\inner}[1]{\underline{#1}}

\newcommand{\refnewtheoremn}[4]{%
\newaliascnt{#1}{#2}
\newtheorem{#1}[#1]{#3}
\aliascntresetthe{#1}
\expandafter\providecommand\csname #1autorefname\endcsname{#4}}

\newcommand{\refnewtheorem}[3]{\refnewtheoremn{#1}{#2}{#3}{#3}}

\newtheorem{theorem}{Theorem}[section]
\refnewtheorem{proposition}{theorem}{Proposition}
\refnewtheorem{lemma}{theorem}{Lemma}
\refnewtheorem{corollary}{theorem}{Corollary}

\theoremstyle{definition}
\refnewtheorem{definition}{theorem}{Definition}
\refnewtheorem{question}{theorem}{Question}
\refnewtheorem{remark}{theorem}{Remark}
\refnewtheorem{example}{theorem}{Example}
\refnewtheorem{notation}{theorem}{Notation}
\refnewtheorem{convention}{theorem}{Convention}
\refnewtheorem{construction}{theorem}{Construction}
\refnewtheorem{claim}{theorem}{Claim}

\begin{document}

\title{Tannaka duality revisited}

\begin{abstract}
We establish several strengthened versions of Lurie's Tannaka duality theorem for certain classes of spectral algebraic stacks. Our most general version of Tannaka duality identifies maps between stacks with exact symmetric monoidal functors between $\infty$-categories of quasi-coherent complexes which preserve connective and pseudo-coherent complexes.
\end{abstract}

\maketitle

\section{Introduction}

\subsection{Background}

The goal of this paper is to investigate algebraic stacks through their associated categories of quasi-coherent sheaves (or, better, complexes). To put this investigation in context, note that an affine scheme is completely determined by its ring of functions. Using merely the ring of functions, though, it is difficult to move beyond affine schemes. However, if one is willing to work with a slightly richer object, one can go much further: any scheme is determined by its (abelian) category of quasi-coherent sheaves by the PhD thesis of Gabriel \cite{gabriel1962categories}. In the last five decades, numerous incarnations of this phenomena --- the encoding of geometric information about a scheme in a linear category associated to that scheme --- have emerged (see \cite{brandenburg2013rosenberg, calabrese2013moduli} for a recent discussion), and have inspired the creation of the subject of noncommutative geometry.

About a decade ago, Lurie introduced a new dimension to this story by adopting a ``relative'' perspective: instead of determining whether or not an isolated geometric object (such as a scheme or an algebraic stack) is determined by a linear category associated to it, he asked when {\em maps} between geometric objects can be reconstructed from the associated ``pullback'' functors on linear categories. More precisely, he showed \cite[Theorem 5.11]{classical_tannaka} that passing to the $\otimes$-category $\QCoh(-)$ of quasi-coherent sheaves is a lossless procedure, even for maps:

\begin{theorem}[Lurie]
\label{thm:LurieClassicalTD}
Let $X$ be a quasi-compact algebraic stack with affine diagonal, and let $S$ be any affine scheme\footnote{One may also take $S$ to be an arbitrary prestack without any gain in generality.}. Then the association $f \mapsto f^*$ gives a fully faithful embedding
\[ \Map(S,X) \to \Fun^L_{\otimes}(\QCoh(X),\QCoh(S)),\]
where the right hand side parametrizes colimit preserving $\otimes$-functors $\QCoh(X) \to \QCoh(S)$. Moreover, the essential image consists exactly of those functors that preserve flatness in a strong sense\footnote{More precisely, one needs the functor to preserve flat objects, as well as monomorphisms with flat cokernels.}.
\end{theorem}

Lurie applied this result in \cite{classical_tannaka} to extend certain algebraization results from Serre's GAGA machinery to stacks. In subsequent work \cite{DAGVIII}, he extended \autoref{thm:LurieClassicalTD} to the case of derived stacks, interpreted in the world of {\em spectral} algebraic geometry. In this homotopical setting, the category $\QCoh(X)$ of quasi-coherent sheaves is an extremely coarse invariant, and one must replace it with the derived $\infty$-category $D(X)$ of quasi-coherent complexes to obtain a well-behaved object. With this caveat, Lurie's spectral Tannaka duality  \cite[Theorem 3.4.2]{DAGVIII} is quite close to the classical one:

\begin{theorem}[Lurie]
\label{thm:LurieSpectralTD}
Let $X$ be a spectral\footnote{The reader unfamiliar with spectral algebraic geometry may safely ignore the adjective ``spectral'' in all the results in the introduction without losing much in terms of content.} algebraic stack with affine diagonal, and let $S$ be any spectral affine scheme. Then the association $f \mapsto f^*$ gives a fully faithful embedding
\begin{equation} \label{eqn:tannaka}
\Map(S,X) \to \Fun^L_{\otimes}(D(X),D(S))
\end{equation}
with essential image spanned by functors that preserve connective objects and flat objects.
\end{theorem}

Here $\Fun^L_{\otimes}(D(X),D(S))$ is the $\infty$-category of colimit preserving $\otimes$-functors $D(X) \to D(S)$, and a complex $K \in D(X)$ is {\em connective} if it lies in $D^{\leq 0}(X)$. In practice, preservation of connectivity is a rather mild constraint, and easy to check. On the other hand, in potential applications of \autoref{thm:LurieSpectralTD}, it is rather hard to verify preservation of flatness. For instance, it is not obvious how to explicitly prove preservation of flat objects in many of the example applications we present below.\footnote{As another example, after constructing a particular symmetric monoidal functor in the proof of Lurie's algebraizability of formal stacks theorem \cite[Theorem 5.4.1]{DAGXII}, checking that it preserves flat objects adds substantial complexity to the argument.}

\subsection{Tannaka duality results} Our goal in this paper is to identify a class of stacks where the flatness condition in \autoref{thm:LurieSpectralTD} can be dropped, or at least be replaced by a more tractable one for applications. As the first step in this direction, we show the following:

\begin{theorem}
\label{thm:CompactlyGeneratedTD}
Let $X$ be a spectral algebraic stack with quasi-affine diagonal such that $D(X)$ is compactly generated. Then, for any spectral affine scheme $S$, the association $f \mapsto f^*$ gives a fully faithful embedding
\begin{equation} 
\Map(S,X) \to \Fun^L_{\otimes}(D(X),D(S))
\end{equation}
with essential image spanned by functors that preserve connective objects.
\end{theorem}

Recall that $D(X)$ is compactly generated if there exist ``enough'' perfect complexes on $X$. In particular, by a fundamental result of Thomason, \autoref{thm:CompactlyGeneratedTD} applies to any (quasi-compact and quasi-separated) algebraic space; this special case, together with applications to algebraization questions, was the subject of \cite{bhatt2014algebraization}, and the proof here is heavily inspired by the latter, though we offer a few simplifications. More generally, \autoref{thm:CompactlyGeneratedTD} applies to a fairly large class of stacks, including quotient stacks in characteristic $0$, and any stack whose diagonal is quasi-finite and separated (by \cite{hall2014perfect}). However, this does not include all stacks one encounters in real life, especially in characteristic $p$. To remedy this, we prove the following:

\begin{theorem}
\label{thm:NoethTD}
Let $X$ be a noetherian spectral algebraic stack with quasi-affine diagonal. Then, for any spectral affine scheme $S$, the association $f \mapsto f^*$ gives a fully faithful embedding
\begin{equation} 
\Map(S,X) \to \Fun^L_{\otimes}(D(X),D(S))
\end{equation}
with essential image spanned by functors that preserve connective objects and pseudo-coherent objects.
\end{theorem}

Recall that $K \in D(X)$ is {\em pseudo-coherent} if, locally on $X$, it can be represented by a bounded-above complex of vector bundles; the relevance of pseudo-coherence here is that noetherian stacks might not always possess ``enough'' perfect complexes, but there is always an abundance of pseudo-coherent complexes.

In addition to these two main theorems, we establish several variants which are easier to apply in practice. For instance, for Noetherian stacks we can identify maps $S \to X$ directly with symmetric monoidal functors $\APerf(X)^{cn} \to \APerf(S)^{cn}$ which preserve finite colimits. We refer the reader to \S \ref{sect:small_categories} and also to \S \ref{sect:DM_stacks} for variants relevant to Deligne-Mumford stacks.

\subsection{Applications to algebraization and representability}

The preceding results have applications to algebraization questions. For example, in analogy with the case of algebraic spaces treated in \cite{bhatt2014algebraization}, one obtains the following continuity, or {\em integrability}, result for adic points:

\begin{corollary}
\label{cor:AlgebraizeLimits}
Let $X$ be a Noetherian spectral algebraic stack with quasi-affine diagonal or a spectral Deligne-Mumford stack whose underlying classical stack is quasi-compact, separated, and tame. Let $A$ be a commutative ring that is $I$-adically complete for some ideal $I$. Then
\[ X(A) \simeq \lim X(A/I^n).\]
\end{corollary}

The main non-trivial step in \autoref{cor:AlgebraizeLimits} is to algebraize a compatible system $\{\Spec(A/I^n) \to X\}$ to a  map $\Spec(A) \to X$. We do this by reformulating the question purely in terms of the derived $\infty$-category (via \autoref{thm:CompactlyGeneratedTD} and \autoref{thm:NoethTD}), and then using suitable continuity properties of of the latter. More generally, these ideas are useful in studying representability questions, and we illustrate this in \S \ref{sect:mapping_stacks} by proving the following, which is the key step in the proof of the algebraicity of $\Hom$-stacks given in \cite{halpern2014mapping}. 

\begin{corollary}
	\label{cor:IntegrabilityMappingStackSpectral}
	Let $f:X \to S$ and $g:Y \to S$ be finitely presented morphisms of noetherian spectral algebraic stacks. Assume that $f$ is proper\footnote{In fact, we simply need $X \to S$ to be cohomologically proper (see \cite{halpern2014mapping}); this allows many more examples, such as $B(\G_m)$.}, and that $g$ has quasi-affine diagonal. Then the $\Hom$-stack $\Map_S(X,Y)$ is integrable, i.e., $\Map_S(X,Y)(-)$ satisfies the conclusion of \autoref{cor:AlgebraizeLimits}.
\end{corollary}

\begin{remark}
	We emphasize that one of novelties here is the lack of finiteness assumptions on the test objects. Indeed, in noetherian situations, standard techniques based on formal geometry often allow one to prove algebraization results relatively painlessly. For example, the special case of \autoref{cor:AlgebraizeLimits} when both $X$ and $A$ are noetherian and classical (i.e., non-derived) follows immediately from Grothendieck's formal existence theorem (at least when $X$ is separated). Likewise, the special case of \autoref{cor:IntegrabilityMappingStackSpectral} when all objects involved are classical can be deduced relatively quickly from Grothendieck's work on formal smoothness (at least when $Y$ has affine diagonal), see \S \ref{ss:NonDerived}. On the other hand, as in \cite{bhatt2014algebraization}, we do not know how to prove the classical version of \autoref{cor:AlgebraizeLimits} in general without any derived input.
\end{remark}

Finally, it is possible to use the Tannaka duality results to identify certain pushouts. For example, it is possible to deduce Beauville-Laszlo type theorems (as in \cite[\S 1.3]{bhatt2014algebraization}) for stacks satisfying the conditions of \autoref{thm:CompactlyGeneratedTD} or \autoref{thm:NoethTD}. Moreover, one can also describe certain ``strange'' pushouts. Here the word ``strange'' refers to the non-flatness of the glueing maps: one can glue along proper birational maps. We give one representative example here; more general statements can be found in \S \ref{sec:strange_pushout}, and this phenomenon will be developed more thoroughly in \cite{hl_reconstruction}.

\begin{example}
	Let $C_0 \subset \P^3$ and $C_1 \subset \P^3$ be two smooth curves that intersect transversally. For $i \in \{0,1\}$, let $Z_i \to \P^3$ be the blowup of $\P^3$ along $C_i$, and let $X_i \to Z_i$ be the blowup of $Z_i$ along the strict transform of $C_{1-i}$. Then $X_0 \simeq X_1$ as schemes over $\P^3$. Moreover using \autoref{thm:NoethTD}, one can show that the resulting commutative square
	\[ \xymatrix{ X_0 \simeq X_1 \ar[d] \ar[r] & Z_1 \ar[d] \\
	Z_0 \ar[r] & \P^3 }\]
	is a pushout square.
\end{example}

\begin{remark}
For additional applications, we refer the reader to \cite{halpern2014on}, where the Tannaka duality results above is used, among other places, in establishing the convexity of the degeneration space of a point in a $\Theta$-reductive stack.
\end{remark}

\subsection{A key intermediate result}

We end this introduction by commenting briefly on a key intermediate result that allows us to relax the hypothesis on the diagonal in Lurie's \autoref{thm:LurieSpectralTD}: we give a description quasi-affine morphisms in terms of the derived $\infty$-category. To motivate this, fix some scheme (or stack) $X$. It is then well-known that the association $Y \mapsto f_* \calO_Y$ gives an equivalence of categories between affine $X$-schemes and commutative algebras in $\QCoh(X)$. It is natural to wonder if this ``algebraic'' classification of ``geometric'' objects can be extended to a wider class of morphisms using $D(X)$ in lieu of $\QCoh(X)$. Lurie has already shown that the same formula, when interpreted at the derived level, gives a fully faithful embedding of the category of quasi-affine morphisms over $X$ into commutative algebras in $D(X)$; we complete the picture by describing the essential image of this functor purely algebraically. Instead of giving the full result here, we simply state the most important special case:

\begin{theorem}
\label{thm:QuasiAffineChar}
Let $X$ be a spectral algebraic stack. Then the association $U \mapsto f_* \calO_U$ identifies the poset of quasi-compact open substacks $U \subset X$ and with the poset of commutative algebras $\calA \in D(X)$ satisfying:
\begin{enumerate}
\item The structure morphism $\calO_X \to \calA$ is a localization (i.e., $\calA \otimes \calA \simeq \calA$ via the multiplication map).
\item $\calA$ is compact as a commutative algebra in $D(X)$, and bounded above as a complex.
\end{enumerate}
\end{theorem}

We emphasize here that $\calA$ is required to be compact as a commutative algebra, and not a complex, in $D(X)$. For a concrete example, consider $X = \A^2$, and let $j:U \to X$ be the complement of $0$. Then $j_* \calO_U$ is extremely large as a complex (the first cohomology group is infinite dimensional). Nevertheless, \autoref{thm:QuasiAffineChar} implies that $j_* \calO_U$ is compact as an algebra over $\calO_X$; an explicit presentation is given in \autoref{ex:PresentationQCOpen}, and can also be found in \cite[Proposition 8.9]{MathewResidueClass}.

\subsection{Relation to existing work}

This paper builds on the ideas of \cite{bhatt2014algebraization} and \cite{DAGVIII}. During the preparation of this paper, Hall and Rydh released the preprint \cite{hall2014coherent}. The methods of proof are different, and the main results are different in that \cite{hall2014coherent} deals with classical stacks and symmetric monoidal abelian categories, whereas below we work in the more general context of spectral algebraic stacks and symmetric monoidal stable $\infty$-categories. The results for Noetherian classical stacks approximately imply one another.\footnote{There are two peculiar differences: in \cite{hall2014coherent} they are able to treat Noetherian stacks with affine stabilizers (slightly weaker than our quasi-affine diagonal hypothesis), and our results classify maps from arbitrary affine schemes and hence arbitrary prestacks (whereas \cite{hall2014coherent} only classifies maps out of locally Noetherian stacks in the quasi-affine diagonal case).} There is also recent work of Sch\"{a}ppi \cite{SchappiConstructingColimits,SchappiDescentTD} interpolating between the two: he establishes algebraization and Tannaka duality statements for possibly non-noetherian stacks admitting enough vector bundles without passing to the derived world.

\subsection{Acknowledgements} We would like to thank Aise Johan de Jong, Jacob Lurie, Akhil Mathew, and David Rydh for helpful conversations regarding this work. The first author was supported by NSF grants DMS 1340424 and DMS 1128155 and the Institute for Advanced Study, while the second author was supported by an NSF postdoctoral fellowship and the Institute for Advanced Study.

\subsection{Notation and conventions}

Throughout this paper, we work in the context of spectral algebraic geometry. A prestack is a presheaf of $\infty$-groupoids on the $\infty$-category of $E_\infty$-rings, and stacks are prestacks which satisfy \'{e}tale descent. We adopt the foundations of \cite{LurieHA}, and for the most part we conform with the notation there and in the DAG papers. In particular, we write $\Perf(X)$ and $\APerf(X)$ for the $\infty$-categories of perfect and almost perfect (i.e., pseudo-coherent) complexes on a stack $X$. One notable exception is that we use $D(X)$ to denote the $\infty$-category of quasi-coherent complexes on a prestack $X$, instead of the notation $\QCoh$, which we reserve for the abelian category of quasi-coherent sheaves. We also use cohomological grading conventions rather than homological ones. We use the notation $\Func(-,-)$ to denote functors preserving finite colimits, and $\Func_\otimes(-,-)$ for symmetric monoidal functors preserving finite colimits.

\section{A purely algebraic description of quasi-affine morphisms}
\label{sec:QuasiAffine}

All stacks in this section are fpqc-stacks on the category of connective $E_\infty$-rings. We will study stacks of the following sort:

\begin{definition}
	A stack $X$ is {\em an fpqc-algebraic stack} if there is a faithfully flat qcqs map $\Spec(A) \to X$ where $A$ is a connective $E_\infty$-ring. 
\end{definition}

	Let $\QAff_{/X}$ be the $\infty$-category of quasi-affine maps $U \to X$, and let $\Op_{qc}(X) \subset \QAff_{/X}$ be the full subcategory spanned by quasi-compact open subsets of $X$; for $U \in \QAff_{/X}$, we write $\calO_U \in \CAlg(D(X))$ for the pushforward of the structure sheaf of $U$. Our goal is to describe $\QAff_{/X}$ and $\Op_{qc}(X)$ in terms of $\CAlg(D(X))$. For this, we make the following definition:

\begin{definition}
	Let $\calC$ be a presentable symmetric monoidal $\infty$-category where the tensor product is compatible with colimits. Given $\calA \in \CAlg(\calC)$, and a commutative $\calA$-algebra $\calB \in \CAlg(\calC)_{\calA/}$, we call $\calB$ a {\em localization of $\calA$} if the multiplication map induces $\calB \otimes_{\calA} \calB \simeq \calB$; if additionally $\calB$ is compact in $\CAlg(\calC)_{\calA/}$, then we say $\calB$ is a {\em compact-localization of $\calA$}. 
\end{definition}

The main result is that these notions capture geometric behaviour. We equip $\CAlg(D(X))$ and $\QAff^{\opp}_{/X}$ with the cocartesian symmetric monoidal structure \cite[Construction 2.4.3.1]{LurieHA}.
\begin{theorem}
	\label{thm:QAchar}
	Say $X$ is an fpqc-algebraic stack. Then:
	\begin{enumerate}
		\item (Lurie\footnote{In fact this holds for any prestack such that $D(X)$ is presentable.}) The association $U \mapsto \calO_U$ defines a fully faithful symmetric monoidal functor $\QAff_{/X}^\opp \to \CAlg(D(X))$.
		\item The essential image of $\QAff_{/X}^\opp \to \CAlg(D(X))$ is given by bounded above $\calA \in \CAlg(D(X))$ which are compact-localizations of some  connective $\calA' \in \CAlg(D(X))$.
		\item The essential image of $\Op_{qc}(X)^\opp \to \CAlg(D(X))$ is given by bounded above $\calA \in \CAlg(D(X))$ which are compact-localizations of $\calO_X$.
	\end{enumerate}
\end{theorem}

We say that $\calA \in \CAlg(D(X))$ is a {\em quasi-affine algebra} if it satisfies the condition in \autoref{thm:QAchar} (2) above, i.e., $\calA \simeq \calO_U$ for some $U \in \QAff_{/X}$; such an $\calA$ is a compact-localization of $\calA' := \tau^{\leq 0} \calA$.

\subsection{Proof of \autoref{thm:QAchar}}

\begin{lemma}
Let $X$ be a prestack such that $D(X)$ is presentable. Then the association $U \mapsto \calO_U$ extends to a fully faithful symmetric monoidal functor $\QAff^{\opp}_{/X} \to \CAlg(D(X))$.
\end{lemma}
\begin{proof}
If $X$ is a prestack such that $D(X)$ is presentable, then the proof of \cite[Proposition 3.2.9]{DAGVIII} shows that $\QAff^{\opp}_{/X} \to \CAlg(D(X))$ is a fully faithful embedding of $\infty$-categories. To show that it is symmetric monoidal with respect to the cocartesian symmetric monoidal structure, it suffice to show that for two quasi-affine maps $Y_0, Y_1 \to X$, the canonical map $\calO_{Y_0} \otimes \calO_{Y_1} \to \calO_{Y_0 \times_X Y_1}$ is an equivalence. Because pushforward along quasi-affine maps satisfies base change \cite[Corollary 3.2.6]{DAGVIII}, it suffices to verify that the map is an equivalence when $X$ is affine. The claim holds for affine spectral schemes because the Yoneda embedding preserves limits. The Cech complex for a finite open cover of $Y_i$ expresses $\calO_{Y_i}$ as a finite limit of the connective $E_\infty$ rings $\calO_{U_{i;\alpha}}$. As tensor products commute with finite limits, one can identify $\calO_{Y_0} \otimes \calO_{Y_1}$ with the Cech complex for the product open cover of $\calO_{Y_0 \times_X Y_1}$.
\end{proof}

\begin{lemma}
Let $X$ be an fpqc-algebraic stack. Fix $\calA \in \CAlg(D(X))$. Then the following are equivalent:
\begin{enumerate}
	\item The forgetful functor $D(X,\calA) \to D(X)$ is fully faithful, i.e. $\calA \otimes \bullet : D(X) \to D(X,\calA)$ is a localization. 
	\item The map $\calO_X \to \calA$ is an epimorphism. 
	\item For any $\calB \in \CAlg(D(X))$, the space $\Map(\calA,\calB)$ is $(-1)$-truncated.
	\item $\calA$ is a localization of $\calO_X$.
\end{enumerate}
Moreover, for such $\calA$, one has $L_{\calA/\calO_X} \simeq 0$.
\end{lemma}
\begin{proof}
	(1) $\Longleftrightarrow$ (2): (1) is equivalent to asking that the natural map $K \to K \otimes \calA$ is an equivalence for any $K \in D(X,\calA)$. This is equivalent to requiring that $\calA \simeq \calA \otimes \calA$ via the first coprojection map, i.e. as a left $\calA$-module, which is equivalent to (2). 
	
	(3) $\Longleftrightarrow$ (2): (3) is equivalent to asking that $\Map(\calA,-) \to \ast$ is a monomorphism of space-valued functors on $\CAlg(D(X))$ which is the definition of (2)
	
	(2) $\Longleftrightarrow$ (4): For any $\infty$-category $\calC$, a map $X \to Y$ in $\calC$ is an epimorphism if and only if the fold-map $Y \sqcup_X Y \to Y$ is an isomorphism. Applying this to $\calC := \CAlg(D(X))$ and noting that coproducts are given by tensor products then proves the claim.

	To see the claim about the cotangent complexes, one uses the Kunneth formula: if $\calO_X \to \calA$ is an epimorphism, then $L_{\calA/\calO_X} \simeq L_{\calA/\calO_X} \oplus L_{\calA/\calO_X}$ via the sum map, and hence $L_{\calA/\calO_X} \simeq 0$. 
\end{proof}

\begin{lemma}
	Let $X$ be an fpqc-algebraic stack. For $U \in \Op_{qc}(X)$, the algebra $\calO_U \in \CAlg(D(X))$ is a compact-localization of $\calO_X$ and eventually connective.
\end{lemma}

\begin{proof}
	All assertions except compactness are clear. Let $D_Z(X)$ denote the kernel of $D(X) \to D(U)$. Then descent shows that the fully faithful inclusion $D_Z(X) \to D(X)$ has a right adjoint $\underline{\Gamma}_Z:D(X) \to D_Z(X)$ defined via $\underline{\Gamma}_Z(K) := \ker(K \to K \otimes \calO_U)$. In particular, there is a triangle
	\[ \underline{\Gamma}_Z(\calO_X) \to \calO_X \to \calO_U. \]
	As $\calO_X \to \calO_U$ is a localization, the space $\Map(\calO_U,\calA)$ is $(-1)$-truncated for any $\calA \in \CAlg(D(X))$. In fact, $\Map(\calO_U,\calA)$ is contractible exactly when $\underline{\Gamma}_Z(\calO_X) \otimes \calA \simeq 0$: if $\underline{\Gamma}_Z(\calO_X) \otimes \calA \simeq  0$, then $\calA \simeq \calA \otimes \calO_U$ is an $\calO_U$-algebra, and conversely if $\calA$ is an $\calO_U$-algebra, then $\calA \otimes \underline{\Gamma}_Z(\calO_X) \simeq \calA \otimes_{\calO_U} \calO_U \otimes \underline{\Gamma}_Z(\calO_X) \simeq \calA \otimes_{\calO_U} 0 \simeq 0$. Now to check compactness of $\calO_U$, we must check: given a filtering system $\{\calA_i\}$ in $\CAlg(D(X))$, if $\Map(\calO_U,\colim \calA_i) \neq \emptyset$, then $\Map(\calO_U,\calA_i) \neq \emptyset$ for $i \gg 0$. Fix such a system $\{\calA_i\}$ with colimit $\calA := \colim \calA_i$, and a map $\calO_U \to \calA$; we want a map $\calO_U \to \calA_i$ for $i \gg 0$ (which then necessarily factors the original map $\calO_U \to \calA$, up to contractible ambiguity). Using the criterion for non-emptyness of $\Map(\calO_U,-)$, the assumption translates to $\underline{\Gamma}_Z(\calO_X) \otimes \calA \simeq 0$, and we want to show $\underline{\Gamma}_Z(\calO_X) \otimes \calA_i \simeq 0$ for $i \gg 0$. As the hypothesis of the last statement passes up along an fpqc cover, and the desired vanishing conclusion can be detected fpqc locally, we may assume $X = \Spec(A)$ is an affine scheme. In this case, there is a perfect complex $K \in D_\perf(X)$ such that $\langle K \rangle = D_Z(X)$. In particular, for any $L \in D(X)$, the vanishing of $\underline{\Gamma}_Z(\calO_X) \otimes L \simeq 0$ is equivalent to the vanishing of $K \otimes L$. We are thus reduced to showing: if $K \otimes \calA \simeq 0$, then $K \otimes \calA_i \simeq 0$ for $i \gg 0$; this comes from the next lemma.
\end{proof}

\begin{lemma}
	Let $X$ be an fpqc-algebraic stack. Fix some $K \in D_\perf(X)$ and a filtering system $\{\calA_i\}$ in $\CAlg(D(X))$. If $K \otimes \colim \calA_i \simeq 0$, then $K \otimes \calA_i \simeq 0$ for $i \gg 0$.
\end{lemma}
\begin{proof}
The assertion is fpqc local, so we may assume $X$ is an affine spectral scheme. Let $\calA := \colim \calA_i$. As $X$ is affine, the complex $K$ is compact in $D(X)$. The assumption $K \otimes \calA \simeq 0$ then implies that the canonical map $\eta_i:K \to K \otimes \calA_i$ is $0$ for $i \gg 0$. The base change isomorphism $\Map_{\calA_i}(K \otimes \calA_i,L) \simeq \Map(K,L)$ for $L \in D(X,\calA_i)$ is defined by composition with $\eta_i$, so it follows that $K \otimes \calA_i \simeq 0$ by Yoneda.
\end{proof}

\begin{example}
\label{ex:PresentationQCOpen}
	Let $X = \Spec(A)$ be a classical affine scheme. Let $I = (f_1,..,f_r) \subset A$ be an ideal, and let $Z = V(I) \subset X$ with complement $U = X - Z$. Then $\calO_U$ can be described explicitly as a compact object via the following recipe: Let $K$ be the Koszul complex on the $f_i$'s, so there is a natural map $\calO_X \to K$ in $\CAlg(D(X))$; here we use that $K$ has a commutative algebra structure as $A$ is classical (via simplicial commutative rings, for example). By duality, this gives a map $\eta:K^\vee \to \calO_X$ in $D(X)$, and hence one has two maps $a,b:\Sym(K^\vee) \to \calO_X$ in $\CAlg(D(X))$ induced by $\eta$ and $0$ on $K^\vee$ respectively. We claim that the compact object $\calA := \calO_X \otimes_{\Sym(K^\vee)} \calO_X \in \CAlg(D(X))$, where the two structure maps $\Sym(K^\vee) \to \calO_X$ are $a$ and $b$, realizes $\calO_U$. To see this, we first show that for any $\calB \in \CAlg(D(X))$, the space $\Map(\calA,\calB)$ is either empty or contractible. Assuming this space is non-empty, fix a map $\calA \to \calB$. By construction, this data gives a nullhomotopy of the composite $K^\vee \to \calB$, and hence a nullhomotopy of the structure map $\calO_X \to K \otimes \calB$; the latter map is an $\calO_X$-algebra map, so the ring structure implies that $K \otimes \calB \simeq 0$. Since $K \otimes \calB \simeq 0$, the space $\Map(\Sym(K^\vee),\calB)$ is contractible, and hence so is $\Map(\calA,\calB) \simeq \Omega \Map(\Sym(K^\vee),\calB))$. This proves that $\calO_X \to \calA$ is an epimorphism, so it is enough to find maps $\mu:\calA \to \calO_U$ and $\nu:\calO_U \to \calA$ in $\CAlg(D(X))$. The existence of $\mu$ follows from the contractibility of $\Map(\Sym(K^\vee),\calO_U)$ (by the same reasoning as above). To get $\nu$, it is enough to show that $\calA \otimes \underline{\Gamma}_Z(\calO_X) \simeq 0$; this is a consequence of $K \otimes \calA \simeq 0$, which itself follows from the reasoning above.
\end{example}

\begin{lemma}[Hopkins-Neeman]
Let $X$ be a classical Noetherian scheme. Let $\calA \in \CAlg(D(X))$ be a localization of $\calO_X$.  Then $\calA \simeq \colim \calO_{U_s}$ for some co-filtered system of open subsets $\{U_s\}$.
\end{lemma}
\begin{proof}
	Because $\calA$ is a localization of $\calO_X$, the kernel, $\calC$, of the base change $F:D(X) \to D(X,\calA)$ consists of those complexes, $M$, for which $\calA \otimes M \simeq 0$. Thus $\calC$ is closed under arbitrary colimits. By the Hopkins-Neeman theorem, we have $\calC = \langle \calC \cap D_\perf(X) \rangle$, i.e., $\calC$ is generated by a set $S$ of perfect complexes $K_s \in \calC$. Let $Z_s$ be the support of $K_s$, and let $U_s := X - Z_s$. Because $K_s$ generates $D_{Z_s}(X)$ \cite{thomason1990higher}, we may discard multiple indices which have the same $U_s$, and because $K_s \otimes K_{s'} \in \calC$, we may enlarge the index set to a sub-poset of the lattice of open subschemes which is closed under intersection.

	The fact that $\calO_{U_s} \otimes \calO_{U_{s'}} \simeq \calO_{U_s \cap U_{s'}}$ implies that if we let $\calO_U := \colim \calO_{U_s}$, then $\calO_U \otimes \calO_{U_s} \simeq \calO_U$ and thus $\calO_U \otimes \calO_U \simeq \calO_U$. Let $\calC'$ be the kernel of $D(X) \to D(X,\calO_U)$.  As $K_s \otimes \calO_{U_{s'}} \simeq 0$ for any $U_{s'} \subset U_s$, we have $K_s \in \calC'$, and so $\calC \subset \calC'$. Moreover, note that $\calC' = \langle \calC' \cap D_\perf(X) \rangle$ by the Hopkins-Neeman theorem. If $L \in \calC' \cap D_\perf(X)$, then $L \otimes \calO_U \simeq 0$. By compactness of $L$, it follows that the canonical map $L \to L \otimes \calO_{U_s}$ is null-homotopic for some $s$ and thus $L \otimes \calO_{U_s} \simeq 0$, so $L \in \langle K_s \rangle \subset \calC$. Hence, $\calC' \subset \calC$, and thus $\calC = \calC'$. Passing the quotients shows that the identity on $D(X)$ restricts to a symmetric monoidal equivalence $D(X,\calA) = D(X,\calO_U)$, and hence $\calA \simeq \calO_U$.
\end{proof}

\begin{lemma}
Let $X$ be an fpqc-algebraic stack. Let $\calA \in \CAlg(D(X))$ be a compact-localization of $\calO_X$. If $\calA$ is eventually connective, then $\calA \simeq \calO_U$ for some quasi-compact open subset $U \subset X$.
\end{lemma}
\begin{proof}
Both $\calA$ and $\calO_U$ are localizations of $X$. Hence, the (fpqc) sheaf of commutative $\calO_X$-algebra isomorphisms between them is either $\emptyset$ or a point. We wish to show it is the latter, which can be checked fpqc-locally, so we may assume $X = \Spec(R)$ for a connective $E_\infty$-ring $R$. Write $A = \Gamma(X,\calA)$, so $A$ is a compact-localization of $R$.  

 We first give the argument when $R$ is discrete.  In this case, we can write $R$ as the filtered colimit $\colim R_i$ of all its finitely generated $\Z$-subalgebras. Let $\CAlg^c(R)$ denote the $\infty$-category of compact objects in $\CAlg(R)$. By \cite[Lemma 7.3.5.16]{LurieHA} and \cite[Lemma 6.24]{MathewGaloisGroup}, one has $\CAlg^c(R) \simeq \colim \CAlg^c(R_i)$, so $A \simeq A_i \otimes_{R_i} R$ for some $A_i \in \CAlg^c(R_i)$. Moreover, by the same argument and increasing $i$ if necessary, we may also assume that $A_i$ is a compact-localization of $R_i$. Replacing $R$ with $R_i$ and $A$ with $A_i$, we may thus assume $R$ is a noetherian ring. In this case, the previous lemma shows that $\calA \simeq \calO_U$ for a pro-(quasi-compact open) subset $U \subset X$. Write $U := \lim U_s$ as a cofiltered limit of open subsets $U_s \subset X$, so $\calO_U \simeq \colim \calO_{U_s}$. In particular, we have $\calO_X$-algebra maps $\calO_{U_s} \to \calA$ for all $s$. By compactness of $\calA$, we also obtain a map $\calA \to \calO_{U_s}$ of $\calO_X$-algebras for some $s$. All these maps are epimorphisms, so $\calA \simeq \calO_{U_s}$.

For a general $E_\infty$-ring $R$, we proceed by deformation-theory. Let $R_n := \tau_{\leq n} R$ be the $n$-th Postnikov truncation of $R$, so $R \simeq \lim R_n$; write $X_n := \Spec(R_n)$ for the corresponding affine spectral scheme, and use a subscript of $n$ to indicate base change along $R \to R_n$. Then $\calA$ defines compact-localizations in $\calA_n$ of $\calO_{X_n}$ for each $n$. For $n = 0$, the previous paragraph gives a quasi-compact open subscheme $U_0 \subset X_0$ such that $\calA_0 \simeq \calO_{U_0}$ in $\CAlg(D(X_0))$. Using the fact that $|X|=|X_n|=|X_0|$, we have unique open subschemes $U_n \subset X_n$ and $U \subset X$ lifting $U_0$. Now $L_{\calA/\calO_X} \simeq 0$ and $L_{\calO_U/\calO_X} \simeq 0$, and $\calO_{U_n}$ is eventually connective $\forall n$, so \autoref{lem:DefThyFormallyEtale} below provides a compatible system of isomorphisms $\calA_n \simeq \calO_{U_n}$. It remains to note that both $\calA$ and $\calO_U$ are convergent, i.e., $\calA \simeq \lim \calA_n$, and $\calO_U \simeq \lim \calO_{U_n}$; indeed, this is true for any eventually connective complex on $X$. 
\end{proof}

\begin{remark}
	The assumption that $\calA$ is eventually connective is necessary unless $X$ is classical. For an example, fix some base ring $k$, and consider $R = \Sym_k(k[2])$. Inverting the generator $u \in \pi_2(R)$ gives a commutative compact $R$-algebra $S$ such that $S \otimes_R S \simeq S$. In fact, we may realize $S$ as the pushout of 
	\[ \Sym_R(R[-2]) \stackrel{uv - 1}{\gets} \Sym_R(R[0]) \stackrel{\epsilon}{\to} R\]
	where $v \in \pi_{-2}(\Sym_R(R[-2]))$ is the image of $1 \in \pi_{-2}(R[-2])$ under the canonical map $R[-2] \to \Sym_R(R[-2])$, $uv - 1$ is the map on algebras induced by a map $R \to \Sym_R(R[-2])$ of $R$-complexes classifying the element $uv - 1$ in $\pi_0(\Sym_R(R[-2]))$, and $\epsilon$ is the augmentation.
\end{remark}

\begin{lemma}
	\label{lem:DefThyFormallyEtale}
	For any connective $E_\infty$-ring $A$, let $F(A)$ denote the $\infty$-category of eventually connective commutative $A$-algebras $B$ such that $L_{B/A} \simeq 0$. If $f:\widetilde{A} \to A$ is a square-zero extension of connective $E_\infty$-rings, then $f$ induces an equivalence $F(\widetilde{A}) \simeq F(A)$.
\end{lemma}
\begin{proof}
		We sketch a proof. The full faithfulness of $F(\widetilde{A}) \to F(A)$ can be checked directly using the vanishing of the cotangent complex; this does not require the connectivity assumption on $B$. For essential surjectivity, fix some $A \to B$ in $F(A)$. The square-zero extension $\widetilde{A} \to A$ has connective kernel $M$, and is classified by an $A$-module map $L_A \to M[1]$. By the transitivity triangle for $A \to B$ and the assumption that $L_{B/A} \simeq 0$, this gives a $B$-module map $L_B \to M_B[1]$, which corresponds to a square-zero extension\footnote{The map $L_B \to M_B[1]$ may be viewed as a section $B \to B \oplus M_B[1]$, and $\tilde{B}$ is defined as the fiber product of this and the universal section $B \to B \oplus M_B[1]$. The canonical map of $E_\infty$-rings $\tilde{A} \to \tilde{B}$, as well as the commutative square of $E_\infty$-ring maps shown, is induced by the commutativity of the diagram $$\xymatrix{A \ar[r]^-{triv} \ar[d] & A \oplus M[1] \ar[d] & & A \ar[d] \ar[ll]_-{L_A \to M[1]} \\ B \ar[r]^-{triv} & B \oplus M_B[1] && B \ar[ll]_-{L_B \to M_B[1]}}.$$ } $\tilde{B} \to B$ fitting into the following commutative diagram of $\tilde{A}$-modules
	\[ \xymatrix{ M \ar[r] \ar[d] & \widetilde{A} \ar[r] \ar[d] & A \ar[d] \\
	M_B \ar[r] & \widetilde{B} \ar[r] & B}\]
with exact rows, where the right hand square underlies a commutative square of $E_\infty$-ring maps. We must check that $\widetilde{A} \to \widetilde{B}$ lies in $F(\widetilde{A})$. For this, by base change for cotangent complexes, it is enough to check that the canonical $A$-algebra map $B' := \widetilde{B} \otimes_{\widetilde{A}} A \to B$ is an isomorphism. This map fits into a diagram
	\[ \xymatrix{ M \ar[r] \ar[d] & \widetilde{A} \ar[r] \ar[d] & A \ar[d] \\
	M_{B'}\simeq M \otimes_{\widetilde{A}} \widetilde{B} \ar[d]^a \ar[r] & \widetilde{B} \ar[d]^b \ar[r] & B' \ar[d]^c  \\
	M_B \ar[r] & \widetilde{B} \ar[r] & B,}\]
	where all rows are exact, and the second row is obtained by extending scalars from the first row along $\widetilde{A} \to \widetilde{B}$. The map $b$ is the identity, while $a$ is identified with $\id_M \otimes c$. Staring at cofibres shows $\coker(a)[1] \simeq \coker(c)$. Using the identification of $a$ with $\id_M \otimes c$ shows $\coker(a) \simeq M \otimes \coker(c)$. Thus, we obtain $\coker(c) \simeq \coker(c) \otimes_A M[1]$. As $M$ is connective, this formally implies that $\coker(c) \simeq 0$ as everything in sight is eventually connective; thus, $B' \simeq B$, as wanted.
\end{proof}

\subsection{Consequences of \autoref{thm:QAchar}}

Recall that a map $\Spec(T) \to \Spec(R)$ of affine (connective) spectral schemes is said to be finitely presented if $T$ is compact as an $R$-algebra. A quasi-affine map $V \to \Spec(R)$ is locally finitely presented if every affine open in $V$ is finitely presented over $\Spec(R)$; the general case is defined via base change.

\begin{proposition} \label{prop:compact_algebras}
Let $X$ be an fpqc-algebraic stack. A quasi-affine algebra $\calA \in \CAlg(D(X))$ is compact if and only if the corresponding quasi-affine morphism $V \to X$ is locally finitely presented.
\end{proposition}

\begin{lemma}
Let $X$ be an fpqc-algebraic stack. Fix maps $A \to B \to C$ in $\CAlg(D(X))$. 
\begin{enumerate}
	\item If $B$ is $A$-compact, and $C$ is $B$-compact, then $C$ is $A$-compact.
	\item If $B$ and $C$ are $A$-compact, then $C$ is $B$-compact.
\end{enumerate}
\end{lemma}
\begin{proof}
For (1), fix a filtering system $\{D_i\}$ of $A$-algebras with colimit $D$. Then we have a commutative square
	\[ \xymatrix{ \colim \Map_A(C,D_i) \ar[r]^f \ar[d] & \colim \Map_A(B,D_i) \ar[d]^{\simeq} \\
			\Map_A(C,D) \ar[r]^g & \Map_A(B,D), }\]
			where the equality on the right uses that $A$-compactness of $B$. Fix a point $p \in \colim \Map_A(B,D_i)$ represented by a map $B \to D_{i_0}$ for some $i_0$. After reindexing, we may assume $i_0$ is initial, so $\{D_i\}$ may be viewed as a diagram of $B$-algebras with colimit $D$. The fibre $f^{-1}(p)$ is given by $\colim \Map_B(C,D_i)$ by exactness of filtered colimits. The fibre $g^{-1}(p)$ is $\Map_B(C,D)$, so it is enough to check that the natural map $\colim \Map_B(C,D_i) \to \Map_B(C,D)$ is an equivalence, which follows from the $B$-compactness of $C$.

			For (2), fix a filtered diagram $\{D_i\}$ of $B$-algebras with colimit $D$. We must check $\colim \Map_B(C,D_i) \simeq \Map_B(C,D)$. Consider the commutative square
			\[ \xymatrix{ \colim \Map_A(C,D_i)  \ar[r] \ar[d]^\simeq & \colim \Map_A(B,D_i) \ar[d]^\simeq \\ 
					\Map_A(C,D) \ar[r] & \Map_A(B,D), }\]
					where the vertical maps are equivalences due to the assumption on $B$ and $C$ respectively. The claim now follows by comparing the fibre of the two horizontal maps over the given point $p \in \Map_A(B,D)$, and using the exactness of filtered colimits.
\end{proof}

\begin{proof}[Proof of \autoref{prop:compact_algebras}]
	Assume first that $V \to X$ is a locally finitely presented quasi-affine morphism. We must show $\calA := \calO_V$ is compact in $\CAlg(D(X))$. This claim is local on $X$, so we may assume $X = \Spec(R)$ for some connective $E_\infty$-ring $R$. Then $V$ is given by a quasi-compact open subset of some $\Spec(T) \to \Spec(R)$ for some $T \in \CAlg_{R/}$. We can write $T$ as a filtered colimit $T \simeq \colim T_i$ with each $T_i \in \CAlg_{R/}$ compact. After possibly replacing the filtering system $\{T_i\}$ with a cofinal subset, we may assume that $V$ comes via pullback from a pullback-compatible system of quasi-compact open subsets $V_i \subset \Spec(T_i)$ for all $i$. In particular, the transition maps $V_i \to V_j$ are affine, and $V \simeq \lim V_i$. As each $\calO_{V_i}$ is compact in $\CAlg(D(X))$ (by transitivity of compactness), it is enough to check that $\calO_V$ is compact as a $\calO_{V_i}$-algebra for some $i$. This assertion can be checked locally on $\calO_{V_i}$, so we reduce to the case where $V_i$ is affine. Then $V$ is also affine, and the claim follows from both $V$ and $V_i$ being finitely presented over $X$.
	
	Conversely, fix a quasi-affine algebra $\calA \in \CAlg(D(X))$ that is compact. Then $\calA = \calO_V$ for some $V \to X$ quasi-affine. To show $V \to X$ is locally finitey presented, we may assume $X$ is affine. For every quasi-compact open subset $W \subset V$, the algebra $\calO_W$ is compact over $\calA$ and thus over $\calO_X$ as well. In particular, if $W$ is itself affine, then $W \to X$ is finitely presented, which proves the claim.
\end{proof}

\begin{remark}
Applying \autoref{prop:compact_algebras} when $X = \Spec(k)$ for some base field provides many examples of compact algebras in $D(\Spec(k))$ which violate the naive expectation that finitely presented algebras should have finite dimensional homology, as is true for finitely presented connective algebras. The question of which stacks have compact algebras of global sections, $\Gamma(X,\calO_X)$, is a subtle one. For instance, we have seen that $\Gamma(\A^2-\{0\},\calO_{\A^2-\{0\}})$ is compact, but $\Gamma(\P^2-\{0\},\calO_{\P^2-\{0\}})$ is the free $E_\infty$-algebra on $\Gamma_{\{0\}}(\A^2,\calO_{\A^2})$, which is an infinite dimensional vector space in degree -1, so this algebra is not compact.
\end{remark}

\begin{corollary}
	Let $X$ be an fpqc-algebraic stack. A quasi-affine algebra $\calA \in \CAlg(D(X))$ comes from an \'etale quasi-affine map $V \to X$ if and only if $\calA$ is compact in $\CAlg(D(X))$ and $L_{\calA/\calO_X} \simeq 0$.
\end{corollary}
\begin{proof}
	We already know that compact quasi-affine algebras in $D(X)$ correspond to locally finitely presented quasi-affine maps $V \to X$. It is then enough to note that such a map is \'etale if and only if $L_{V/X} \simeq 0$, and that $L_{V/X}$ corresponds to $L_{\calO_V/\calO_X}$ under the equivalence $D(V) \simeq D(X,\calO_V)$.
\end{proof}

\begin{remark}
	The material in this section suggests it might be worthwhile to investigate the following class of maps: say a map $f:X \to Y$ of schemes (or stacks) is {\em weakly quasi-affine} if $f_*$ establishes an equivalence $D(X) \simeq D(Y, f_* \calO_X)$. It is clear that quasi-affine maps belong to this class. More generally, one can show that any map $f$ that is affine over a constructible stratification of the target is also weakly quasi-affine. This leads one to ask: is there a geometric classification of weakly quasi-affine maps? Do the Tannaka duality results in this article extend to stacks whose diagonal is only assumed to be weakly quasi-affine?
\end{remark}

\section{tannakian stacks}

\begin{definition} \label{defn:tannakian}
For an arbitrary prestack, $X$, we say that $X$ is \emph{tannakian} if
\[ \Map(S,X) \to \Fun^L_{\otimes}(D(X),D(S))\]
is fully faithful with essential image given by functors preserving almost perfect complexes and connective complexes.
\end{definition}

\begin{remark}
If $X$ is an algebraic stack, then an equivalent definition of tannakian is that $\Map(S,X) \to \Fun^L_\otimes(D(X)^{cn},D(S)^{cn})$ is fully faithful with essential image given by functors preserving almost perfect complexes: Because $D(X)$ is right $t$-complete, $\Fun^L_{\otimes}(D(X)^{cn},D(S)^{cn})$ can be identified with the full subcategory of $\Fun^L_\otimes(D(X),D(S))$ consisting of functors which are right $t$-exact \cite[Lemma 5.4.6]{DAGXII}. This also identifies the subcategories of functors which preserve almost perfect complexes.
\end{remark}

\begin{remark}
We will see below that for a Noetherian stack which is tannakian, $\Map(S,X)$ is also equivalent to the $\Func_\otimes (\APerf(X)^{cn},\APerf(S)^{cn})$ of symmetric monoidal functors which preserve finite colimits.
\end{remark}

Any prestack is a colimit of affine prestacks, so to show that a prestack is tannakian, it suffices to consider affine $S$ in the definition above. Proposition 3.3.11 of \cite{DAGVIII} states that when $X$ is quasi-geometric, then $\Map(S,X) \to \Fun^L_{\otimes}(D(X),D(S))$ is fully faithful. Thus for quasi-geometric stacks the goal is to describe the essential image of this map. Given a cocontinuous symmetric monoidal functor $F : D(X) \to D(S)$, we will say that $F$ is \emph{geometric} if it corresponds to a map $S \to X$.

\begin{proposition} \label{prop:covering}
Let $X$ be a quasi-compact quasi-geometric algebraic stack. A cocontinuous symmetric monoidal functor $F : D(X) \to D(S)$ is geometric if and only if there are fpqc quasi-affine maps $U \to X$ and $V \to S$, with $U$ affine, and a map of $\calO_S$-algebras $\calO_V \to F(\calO_U)$.
\end{proposition}

\begin{lemma}
	\label{lem:fpqcDescentQuasiAffine}
	Let $X$ be an algebraic stack for the fppf topology. Assume that $X$ has quasi-affine diagonal. Then $X$ is a stack for the fpqc topology.
\end{lemma}
\begin{proof}
Choose an faithfully flat map $A \to B$ of rings with Cech nerve $B^\bullet$, and let $\epsilon_\bullet:\Spec(B^\bullet) \to X$ be a given morphism. We must show that this arises from a unique morphism $\epsilon:\Spec(A) \to X$. Let $f:U \to X$ be an fppf cover of $X$ by an affine scheme. Then $f$ is fppf and quasi-affine by assumption. The base change of $f$ along $\epsilon_\bullet$ defines an fppf quasi-affine cover $U \times_X \Spec(B) \to \Spec(B)$ that is compatible with the descent datum relative to the map $\Spec(B) \to \Spec(A)$. As fppf quasi-affine maps satisfy descent for the fpqc topology, this comes from an fppf quasi-affine map $V \to \Spec(A)$ via base change. Moreover, the fpqc sheaf property for $U$ gives a map $V \to U$. Repeating the construction for fibre products of $U$ over $X$ then gives a map $V^\bullet \to U^\bullet$, where $V^\bullet \to \Spec(A)$ is the Cech nerve of $V \to \Spec(A)$, and $U^\bullet \to X$ is the Cech nerve of $f$. One then concludes by the fppf descent property for $X$.
\end{proof}

\begin{proof}[Proof of \autoref{prop:covering}]
The only if part is immediate, so assume we have $U,V$, and $\calO_V \to F(\calO_U)$ as in the statement of the lemma. One obtains a symmetric monoidal functor $D(U) \to D(S,F(\calO_U))$, and by composition one gets a continuous symmetric monoidal functor $D(U) \to D(X,\calO_V) \simeq D(V)$ lifting $F$. As affine schemes are tannakian, this functor is geometric. Likewise, if one considers the Cech nerve $U_\bullet$ of $U \to X$ and the Cech nerve $V_\bullet / S$, then one obtains a map of cosimplicial objects in symmetric monoidal categories $D(U_\bullet) \to D(V_\bullet)$ which is geometric at every level, and hence corresponds to a map of simplicial stacks $V_\bullet \to U_\bullet$. By fpqc descent (for quasi-coherent sheaves as well as for maps of stacks) on obtains a morphism $S \simeq |V_\bullet| \to |U_\bullet| \simeq X$ inducing the functor $F : D(X) \simeq \op{Tot} (D(U_\bullet)) \to \op{Tot} (D(V_\bullet)) \simeq D(S)$. (This is essentially the idea behind the proof of Lurie's tannakian formalism for geometric stacks \cite[Theorem 3.4.2]{DAGVIII}).
\end{proof}

Later we will use this lemma to prove that perfect stacks (\autoref{thm:td_perfect_stacks}) and Noetherian quasi-geometric stacks (\autoref{thm:td_Noetherian}) are tannakian, but first we collect some direct consequences of this definition.

\subsection{General properties of tannakian stacks.} Following \cite[Definition 8.21]{DAGVIII}, we will call a closed immersion of algebraic stacks $i : X \to Y$ a \emph{nilpotent thickening} if locally it is given by a map of connective $E_\infty$ rings which is a nilpotent extension on $\pi_0(\bullet)$.

\begin{proposition} \label{prop:nilpotent_thickening}
Let $X \subset Y$ be a nilpotent thickening of quasi-compact quasi-geometric algebraic stacks. If $X$ is tannakian, then $Y$ is tannakian. Furthermore, it suffices to check that for any \emph{classical} ring, $R$, any cocontinuous symmetric monoidal functor $D(X) \to D(\Spec(R))$ which preserves connective and almost perfect complexes is geometric.
\end{proposition}

\begin{lemma} \label{lem:characterize_flatness}
Let $R \to S$ be a nilpotent thickening of $S$. Then a connective $R$-module, $M$, is flat if and only if $S \otimes_R M$ is a flat $S$-module.
\end{lemma}
\begin{proof}
The only if part is immediate, and it suffices to show the other direction separately in the case where $S = \pi_0(R)$ and in the case where both $R$ and $S$ are discrete. First assume $S = \pi_0(R)$. It is enough to show that $N \otimes_R M$ is discrete for any discrete $R$-module $N$. If $N$ is discrete, it comes from a $\pi_0(R)$-module, so $N \otimes_R M \simeq \pi_0(N) \otimes_{\pi_0(R)} (\pi_0(R) \otimes_R M)$, which is discrete if $\pi_0(R) \otimes_R M$ is flat.

Next assume that $R$ and $S$ are discrete, and let $I = \ker(R \to S)$. Then any discrete $R$-module $N$ admits a finite filtration by powers of $I$, so in order to show that $N \otimes_R M$ is discrete, it suffices to show that $(I^k N / I^{k+1}N) \otimes_R M$ is discrete for all $k$. $I^k N / I^{k+1} N$ comes from an $S$-module, so again we have $(I^k N / I^{k+1}N) \otimes_R M \simeq (I^k N / I^{k+1}N) \otimes_S (S \otimes_R M)$ is discrete if $S \otimes_R M$ is flat.
\end{proof}

\begin{proof}[Proof of \autoref{prop:nilpotent_thickening}]
In order to show that a continuous symmetric monoidal functor $F: D(Y) \to D(\Spec(R))$ is geometric, it suffices by \autoref{prop:covering} to show that $F$ preserves quasi-affine algebras corresponding to fpqc quasi-affine maps. The fact that $i:X \to Y$ is a nilpotent thickening implies that $R \to S := F(\calO_X)$ is a nilpotent thickening as well, because if $I = \ker(H_0( \calO_Y) \to H_0(\calO_X))$, then $F(I) \to \ker(H_0(R) \to H_0(S))$ is surjective. Because $Y$ is quasi-compact, the map $I^{\otimes n} \to H_0(\calO_Y)$ vanishes for $n$ large and thus so does $F(I)^{\otimes n} \to H_0(R)$. The composition of $F$ with the symmetric monoidal functor $D(R) \to D(S)$ factors through $D(Y,\calO_X) \simeq D(X)$, and the corresponding functor $F' : D(X) \to D(S)$ preserves connective objects because $F$ does. The functor $F'$ corresponds to a map $\Spec(S) \to X$ because $X$ is tannakian, and thus $F'$ preserves quasi-affine algebras corresponding to fpqc quasi-affine maps $U \to X$. Note that as a consequence of \autoref{thm:QAchar}, if $i : X \to Y$ is a nilpotent thickening and $\calA \in \CAlg(D(Y)^{cn})$, then $i^\ast$ induces an equivalence between the category of quasi-affine localizations of $\calA$ and quasi-affine localizations of $i^\ast \calA$. Furthermore \autoref{lem:characterize_flatness} implies that under this equivalence $\calA \to \calO_U$ corresponds to a fpqc quasi-affine map $U \to Y$ if and only if $i^\ast \calA \to i^\ast \calO_U$ corresponds to an fpqc quasi-affine map. We apply this observation to the nilpotent thickening $i' : \Spec(S) \to \Spec(R)$, along with the equivalence $(i')^\ast \circ F \simeq F' \circ i^\ast$ to conlude that $F$ preserves algebras corresponding to fpqc quasi-affine maps.

The same argument shows that $F : D(X) \to D(\Spec(R))$ preserves fpqc quasi-affine algebras if and only if the composition with $D(\Spec(R)) \to D(\Spec(\pi_0 R))$ preserves fpqc quasi-affine algebras. Hence it suffices to consider only classical rings when checking that $X$ is tannakian.
\end{proof}

\begin{lemma}
	Let $X$ be a quasi-geometric tannakian algebraic stack, and let $\widehat{X}$ be the formal completion of $X$ along the complement of a quasi-compact open substack $U$. Then $\widehat{X}$ is tannakian.
\end{lemma}
\begin{proof}
We first remark that $\widehat{X}$ is quasi-geometric: the diagonal of the monomorphism $\widehat{X} \to X$ is an isomorphism, and hence $\widehat{X}$ is quasi-geometric since $X$ is so (using \cite[Proposition 3.3.3]{DAGVIII}). Furthermore $D(\widehat{X})$ is presentable because presentable categories are closed under small limits.\footnote{By choosing an fpqc hypercover of $X$ by affines and applying fpqc hyperdescent for $D(X)$, we may assume $X = \Spec(A)$ for a connective $E_\infty$-ring $A$, and $\widehat{X}$ is the completion of $X$ along a finitely generated ideal $I \subset \pi_0(A)$. Then \cite[Lemma 5.1.5]{DAGXII} shows that $D(\widehat{X}) \simeq \lim D(\Spec(A_n))$ for a suitable set of rings $A_n$.}

	Using \cite[Proposition 3.3.11]{DAGVIII}, it is enough to check: for each affine $S$ and each cocontinuous symmetric monoidal functor $F:D(\widehat{X}) \to D(S)$ preserving connective complexes and almost perfect complexes is geometric. Any such $F$ induces an analogous functor $F':D(X) \to D(S)$ via composition along the pullback $D(X) \to D(\widehat{X})$. Since $X$ is tannakian, one obtains a unique map $f:S \to X$, up to contractible choices, such that $F' \simeq f^*$. As $\widehat{X} \to X$ is a monomorphism, it is enough to check that $f$ factors through $\widehat{X}$. By general properties of formal completions, this is equivalent to showing that $F'(\calO_U) = 0$, which is obvious since the pullback $D(X) \to D(\widehat{X})$ kills $\calO_U$.
\end{proof}

The following shows that quasi-compact open substacks and closed substacks of a quasi-geometric tannakian stack are tannakian.

\begin{lemma} \label{lem:general_properties}
Let $U \to X$ be a quasi-affine map. If $X$ is tannakian and quasi-geometric and $D(X)$ is presentable, then $U$ is tannakian.
\end{lemma}

\begin{proof}
For the first claim, note that a cocontinuous symmetric monoidal functor $F : D(U) \simeq D(X,\calO_U) \to D(\Spec R)$ induces a cocontinuous symmetric monoidal functor $D(X) \to D(\Spec R)$ which preserves connective and almost perfect complexes if $F$ does. Thus the latter functor is geometric, and the corresponding map $f : \Spec R \to X$ is quasi-affine because $X$ is quasi-geometric. The functor $F$ induces a map $\calO_U \to f_\ast R$ in $\CAlg(D(X))$, and by \autoref{thm:QAchar} this map is geometric, hence induced by a map $\Spec R \to U$.
\end{proof}

\begin{lemma}
Let $X \to Z \gets Y$ be a diagram of quasi-geometric tannakian stacks over a field of characteristic $0$. Assume that $X$ and $Z$ are QCA in the sense of Drinfeld-Gaitsgory \cite{drinfeld2013some}. Then $X \times_Z Y$ is also a quasi-geometric tannakian stack.
\end{lemma}

\begin{proof}
One first uses \cite[Theorems 0.3.3 and 0.3.4]{drinfeld2013some} to show that our diagram satisfies the hypotheses of \cite[Proposition 3.3.3]{GaitsgoryQCohNotes}; the latter implies that base change is an equivalence $D(X) \otimes_{D(Z)} D(Y) \to D(X\times_Z Y)$. $X \times_Z Y$ is quasi-geometric,\footnote{By \cite[Proposition 3.3.3]{DAGVIII}, the map $Y \to Z$ is quasi-geometric as $Y$ and $Z$ are so, and hence the map $X \times_Z Y \to X$ is quasi-geometric, and so $X \times_Z Y$ is quasi-geometric since $X$ is so.} so full faithfulness follows from \cite[Proposition 3.3.11]{DAGVIII}. Given a potentially geometric functor  $D(X \times_Z Y) \to D(S)$,  for an affine $S$, we obtain potentially geometric functors $D(X) \to D(S)$ and $D(Y) \to D(S)$ via composition with with the projections. Since both $X$ and $Y$ are tannakian, this defines a map $S \to X \times_Z Y$, and this map is unique because the natural map $D(X) \otimes_{D(Z)} D(Y) \to D(X \times_Z Y) \to D(S)$ is an equivalence.
\end{proof}

For algebraic stacks, it is necessary to restrict to functors which preserve connective complexes, as the following plausibility argument shows.

\begin{example}
Consider $X = S = B(\G_m)$ over a field $k$ of characteristic $0$. Write $L$ for the universal line bundle on $X$. Then informally speaking (we are not aware of a reference that makes this statement rigorous), $\Perf(X)$ is the free  symmetric monoidal stable $k$-linear category on one invertible object (namely, $L$). Let $F:D(X) \to D(S)$ be the unique cocontinuous symmetric monoidal functor that carries $L$ to $L[-1]$. Then $F$ cannot be geometric, as it fails to preserve connective complexes.
\end{example}

\subsection{Functors between small categories}
\label{sect:small_categories}

In applications of the tannakian formalism, it is often easier to work with categories of almost perfect (or perfect) complexes, rather than considering all quasi-coherent sheaves.

\begin{lemma}
If $X$ is a Noetherian perfect stack and $S$ is any prestack, then a cocontinuous symmetric monoidal functor $F : D(X) \to D(S)$ preserves connective complexes if and only if it preserves connective almost perfect complexes.
\end{lemma}
\begin{proof}
If $F$ preserves connective complexes, then it is geometric by \autoref{thm:td_perfect_stacks}, so it preserves almost perfect complexes. Conversely, we can write any connective $M$ as a filtered colimit $\colim P_i$ of perfect complexes, and hence we also have $M \simeq \colim \tau_{\geq 0} P_i$. It follows that $F^\ast M \simeq \colim F^\ast (\tau_{\geq 0} P_i)$ is also connective.
\end{proof}

\begin{lemma} \label{lem:small_cat_Noetherian}
Let $X$ be a Noetherian stack, and let $S$ be a prestack. If $\Fun_\otimes^L(D(X),D(S))'$ denotes the full subcategory of $\Fun_\otimes^L(D(X),D(S))$ consisting of functors preserving connective complexes and almost perfect complexes, then the restriction map induces an equivalence
\[\Fun_{\otimes}^L(D(X),D(S))' \to \Func_{\otimes} (\APerf(X)^{cn},\APerf(S)^{cn})\]
\end{lemma}

\begin{proof}
We will abuse notation by letting $\Fun_\otimes(\bullet)'$ denote the full subcategory of symmetric monoidal functors which preserve connective and almost perfect complexes whenever this notion is well defined for the categories involved. The restriction $\Fun_{\otimes}^L(D(X),D(S))' \to \Fun_{\otimes}^L(D(X)^{cn},D(S)^{cn})'$ is an equivalence by \cite[Lemma 5.4.6]{DAGXII}, so it suffices to show that
\[\Fun_{\otimes}^L(D(X)^{cn},D(S)^{cn})' \to \Func_{\otimes} (\APerf(X)^{cn},\APerf(S)^{cn})\]
is an equivalence. Because the categories $D(S)^{cn}$ and $\APerf(S)^{cn}$ are defined formally as limits over all maps from affine schemes to $S$, it suffices to consider the case where $S$ is affine. The argument uses the same idea as the proof of \cite[Theorem 5.4.1]{DAGXII}: $D(S)$ is left $t$-complete, so $D(S)^{cn} = \varprojlim D(S)^{cn}_{\leq n}$, and every functor $D(X)^{cn} \to D(S)^{cn}_{\leq n}$ factors uniquely through $D(X)^{cn}_{\leq n}$ up to contractible choices because category $D(S)^{cn}_{\leq n}$ is equivalent to an $n+1$-category and $D(X)^{cn} \otimes_{\Sp} \tau_{\leq n} \Sp \simeq D(X)^{cn}_{\leq n}$, where $\Sp$ and $\tau_{\leq n} \Sp$ denote the symmetric monoidal $\infty$-categories of spaces and $n$-truncated spaces respectively.

We claim that the same is true for $\APerf$: For the affine spectral scheme $S$, $\APerf(S)^{cn} \simeq \lim \APerf(S)^{cn}_{\leq n}$ by definition.\footnote{$S$ is not Noetherian, so the $t$-structure on $D(S)$ does not descend to $\APerf(S)$. We are using the notation $\APerf(S)^{cn}_{\leq n}$ to denote the category of finitely $n$-presented objects of $D(S)$ (which are by definition connective and $n$-truncated), and by definition $\APerf(S)^{cn} \subset D(S)^{cn}$ consists of complexes all of whose truncations are finitely $n$-presented.} On the other hand $X$ is Noetherian, so we have a truncation functor $\APerf(X)^{cn} \to \APerf(X)^{cn}_{\leq n}$ which exhibits the latter as a localization of the former by the subcategory $\calO_X[n+1] \otimes \APerf(X)^{cn} \subset \APerf(X)^{cn}$. It follows that for any symmetric monoidal $\infty$-category $\calC$, restriction induces a fully faithful embedding
$$\Func_\otimes(\APerf(X)^{cn}_{\leq n},\calC) \to \Func_\otimes (\APerf(X)^{cn},\calC)$$
whose essential image consists of those functors annihilating $\calO_X[n+1]$. In particular if $\calC$ has finite colimits and $\mathbf{1}_{\calC}[n+1] \simeq 0$ in $\calC$, then this restriction functor is an equivalence.

The above observations imply that it suffices to show that
\[\Fun_\otimes^L (D(X)^{cn}_{\leq n},D(S)^{cn}_{\leq n})' \to \Func_\otimes (\APerf(X)^{cn}_{\leq n},\APerf(S)^{cn}_{\leq n})\]
is an equivalence for all $n\geq 0$. Because $D(X)^{cn}_{\leq n} = \Ind (\APerf(X)^{cn}_{\leq n})$ (see \cite[?]{halpern2014mapping}), the restriction map $\Fun_\otimes^L(D(X)^{cn}_{\leq n}, D(S)^{cn}_{\leq n}) \to \Func_\otimes(\APerf(X)^{cn}_{\leq n}, D(S)^{cn}_{\leq n})$ is an equivalence, and the claim follows because $\Func_\otimes(\APerf(X)^{cn}_{\leq n}, D(S)^{cn}_{\leq n})' \simeq \Fun^c_\otimes(\APerf(X)^{cn}_{\leq n},\APerf(S)^{cn}_{\leq n})$ by definition.
\end{proof}

\subsection{Deligne-Mumford stacks}
\label{sect:DM_stacks}

For separated classical Deligne-Mumford stacks, one does not even need to use the $t$-structure in the definition of tannakian.
 
\begin{corollary} \label{cor:tame_DM}
Let $X$ be a classical quasi-compact tame separated Deligne-Mumford stack. Then the restriction maps
$$\Map(S,X) \simeq \Fun_\otimes^L (D(X),D(S)) \simeq \Fun_\otimes^L (D(X)^c,\Perf(S))$$
are equivalences of $\infty$-categories for any prestack $S$, where $D(X)^c \subset D(X)$ denotes the subcategory of compact objects.
\end{corollary}
\begin{proof}
For any $X$ satisfying the hypotheses of the theorem, $D(X)$ is compactly generated by \cite{hall2014perfect}, so $X$ is tannakian by \autoref{thm:td_perfect_stacks}. Furthermore, $X$ is perfect because it is tame. It follows that
$$\Fun^L_\otimes(D(X),D(S)) \simeq \Func_\otimes(\Perf(X),D(S)),$$
and any symmetric monoidal functor $\Perf(X) \to D(S)$ factors uniquely through the full subcategory of dualizable objects $\Perf(S) \subset D(S)$. The condition that a symmetric monoidal functor $D(X) \to D(S)$ preserves connective and almost perfect complexes is vacuous by the following lemma.
\end{proof}

\begin{lemma}
Let $X$ be a quasi-compact separated classical Deligne-Mumford stack, and let $S$ be an arbitrary pre-stack. Any cocontinuous symmetric monoidal functor $F : D(X) \to D(S)$ preserves connective complexes and almost perfect complexes.
\end{lemma}

\begin{proof}
Let $p : X \to \bar{X}$ be the coarse moduli space, which exists in our level of generality by \cite{rydh2013existence}. First consider the case where $X$ is a quotient of an affine scheme by a finite group. Then $X$ has enough vector bundles, so every connective complex can be written as a colimit of connective perfect complexes, and every almost perfect complex is equivalent to a right-bounded complex of vector bundles. Thus it suffices to show that $F(E)$ is connective for any vector bundle $E$. $F(E)$ is perfect, because it is dualizable, and thus for any $n >0$, $F(E)$ is connective if and only if $F(E)^{\otimes n} \simeq F(E^{\otimes n})$ is connective.\footnote{Connectivity for perfect complexes can be detected on fibers by Nakayama's lemma, and over a field this claim is immediate.} However, for a sufficiently large $n$, $E^{\otimes n} \simeq p^\ast (\tilde{E}^{\otimes n})$ for some vector bundle $\tilde{E}$ on $\bar{X}$. Thus $F(E^{\otimes n})$ is a vector bundle, because the composition $D(\bar{X}) \to D(X) \to D(S)$ is geometric by \cite{bhatt2014algebraization}.

For a general $X$, let $U \to \bar{X}$ be an \'{e}tale cover by an affine scheme such that the base change $X' := X \times_{\bar{X}} U$ is a global quotient of an affine scheme by a finite group (an affine quotient presentation is provided by the proof, but not the statement, of \cite[Lemma 2.2.3]{abramovich2002compactifying}). Because the composition $F \circ p^\ast : D(\bar{X}) \to D(\Spec(R))$ is geometric, the base change of $U$ is gives a surjective \'{e}tale map $V \to \Spec(R)$, where $V$ is affine, and a symmetric monoidal functor $F' : D(X') \to D(V)$ making the following diagram commute,
$$\xymatrix{D(X) \ar[d]^{p^\ast} \ar[r]^-F & D(\Spec(R)) \ar[d] \\ D(X') \ar[r]^-{F'} & D(V) }.$$
More precisely, $X' \to X$ affine, and because $\bar{X}$ is a quasi-compact separated algebraic space, the functor $F \circ p^\ast$ is geometric and $F(\calO_{X'})$ is a connective $R$-algebra. $V$ is the corresponding affine scheme, and the map $D(X') \to D(V)$ corresponds to the symmetric monoidal functor $D(X,\calO_{X'}) \to D(\Spec(R)), F(\calO_{X'}))$ induced by $F$. It follows from the commutative square that $F$ preserves connective complexes and almost perfect complexes if and only if $F'$ does. So we have reduced to the claim to the case of a global quotient of an affine scheme by a finite group, treated above.
\end{proof}

\section{Stacks with compactly generated derived categories}

For any fpqc-algebraic stack $X$, compact objects in $D(X)$ are perfect and hence dualizable.\footnote{This is a consequence of the fact that given an atlas $\pi : U \to X$, the pushforward functor commutes with filtered colimits. Hence compact objects are locally compact by the adjunction between $\pi_\ast$ and $\pi^\ast$.} Let $D(X)^c \subset D(X)$ be the full subcategory spanned by compact objects. Our goal in this section is to prove a Tannaka duality results for stacks $X$ with the property that  $D(X)$ is {\em compactly generated}, i.e., the natural functor $\Ind(D(X)^c) \to D(X)$ is an equivalence or, equivalently, that each $K \in D(X)$ can be written as a filtered colimit of compact objects.

\begin{theorem} \label{thm:td_perfect_stacks}
Let $X$ be an fpqc-algebraic stack with quasi-affine diagonal such that $D(X)$ is compactly generated. Then 
\[ \Hom(S,X) \to \Fun^L_{\otimes}(D(X),D(S))\]
is fully faithful with essential image given by functors preserving connective complexes.
\end{theorem}
The proof below follows the one given in \cite{bhatt2014algebraization} for algebraic spaces; the argument presented here uses the classification of quasi-affine maps in terms of $D(X)$, and is consequently a bit simpler. 
\begin{proof}[Proof of \autoref{thm:td_perfect_stacks}]
Full faithfulness is \cite[Proposition 3.3.11]{DAGVIII}. For essential surjectivity, fix some $F \in \Fun^L_{\otimes}(D(X),D(S))$. By descent, we may assume $S$ is affine, so compact objects in $D(S)$ coincide with the perfect ones and also the dualizable ones. Then $F$ admits a right adjoint $G:D(S) \to D(X)$ by cocontinuity of $F$. As $D(X)$ is compactly generated, and because compacts are dualizable, it follows that $F$ preserves compactness, so $G$ is cocontinuous. Moreover, using the compact generation of $D(X)$, one checks that the pair $(F,G)$ satisfy a projection formula: $G(F(K) \otimes L) \simeq K \otimes G(L)$ for $K \in D(X)$ and $L \in D(S)$.  Then Barr-Beck-Lurie shows that $G$ induces an equivalence $\Phi:D(S) \simeq D(X,G\calO_S)$ such that $F$ corresponds to the base change functor $-\otimes G\calO_S : D(X) \to D(X,G\calO_S)$.

In order to show that $F$ is geometric, it suffices by \autoref{prop:covering} to show that for a faithfully flat quasi-affine cover $f:U \to X$ with $U$ affine (which exists because $X$ is perfect), $F(\calO_U)$ is a faithfully flat quasi-affine $\calO_S$-algebra. Recall that $\calO_U$ is the compact-localization of a connective algebra $\calA \in \CAlg(D(X))$. It is immediate that $F(\calO_U)$ is an eventually connective localization of the connective algebra $F(\calA)$, so we must verify the compactness of $F(\calO_U)$ as an $F(\calA)$-algebra. For this, note first that $D(X,\calA)$ is compactly generated: any $K \in D(X,\calA)$ is a retract of $K \otimes_{\calO_X} \calA$, and the latter is a filtered colimit of compact objects by hypothesis. It follows that $\CAlg(D(X,\calA))$ is compactly generated as well. Now note that $F$ induces a functor $F_{\calA}:D(X,\calA) \to D(S,F(\calA))$. Moreover, under $\Phi$, this is identified with the base-change functor $D(X,\calA) \to D(X,\calA \otimes G\calO_S)$. The right adjoint of this last functor commutes with filtered colimits because $G$ does. It follows that the induced functor $\CAlg(D(X,\calA)) \to \CAlg(D(S,F(\calA)))$ has a cocontinuous right adjoint, and hence must preserve compactness, so $F(\calO_U)$ is compact over $F(\calA)$. Thus, $F(\calO_U) \simeq \calO_{F(U)}$ for some quasi-affine morphism $F(U) \to S$.

	It remains to show that $F(U) \to S$ is faithfully flat. The ``faithful'' part is clear once we use the projection formula to identify the functor $F(\calO_U) \otimes -$ with $\calO_U \otimes -$ under $\Phi$. An object in $D(F(U)) \simeq D(S,F(\calO_U))$ is coconnective if and only if its pushforward to $S$ (i.e. the underlying object of $D(S)$) is coconnective.\footnote{Indeed, it suffices to show this for a quasi-compact open immersion $j : U \hookrightarrow \Spec(R)$. In this case, for any $M \in D^{<0}(U)$, $j^\ast (\tau^{<0} j_\ast M) \simeq j^\ast j_\ast M \simeq M$, so any connective object is the restriction of a connective object. The claim follows by adjunction and the characterization of $D^{\geq 0}(U)$ as the right orthogonal of $D^{<0}(U)$.} Thus, to show that the quasi-affine morphism $F(U) \to S$ is flat it suffices to show that $F(\calO_U) \otimes -$ preserves $D^{\geq 0}(S)$. In fact for any $K \in D(X)$ with the property that $K \otimes -$ preserves $D^{\geq 0}(X)$, we can show that $F(K) \otimes -$ preserves $D^{\geq 0}(S)$. For $M \in D^{\geq 0}(S)$ consider $G(F(K) \otimes M) \simeq K \otimes G(M)$. The fact that $F(K) \otimes M \in D^{\geq 0}(S)$ is now a consequence of the following:
	
	\begin{claim} 
		$M \in D(S)$ is coconnective if and only if $G(M) \in D(X)$ is coconnective.
	\end{claim}
		
	\begin{proof} $D^{\geq 0}(X)$ is right-orthogonal to $D^{< 0}(X)$, so we want $\Map(K,G(M)) = 0$ for $K \in D^{< 0}(X)$. By adjunction, this space is identified with $\Map(F(K),M)$. Since $F$ preserves connectivity, we deduce that $F(K) \in D^{< 0}(S)$, and thus $\Map(F(K),M) = 0$ by the coconnectivity of $M$. Conversely, fix an $M \in D(S)$ such that $G(M) \in D^{\geq 0}(X)$; we want to conclude $M \in D^{\geq 0}(S)$. As $D^{< 0}(S)$ is generated by $\calO_S[1]$ under colimits, it suffices to show that $\Map(\calO_S[1],M) = 0$. By adjunction, this amounts to $\Map(\calO_X[1],G(M)) = 0$, which follows from the hypothesis on $M$. \end{proof}  \qedhere \end{proof}

\section{Noetherian quasi-geometric stacks}

The goal of this section is to prove the following:

\begin{theorem} \label{thm:td_Noetherian}
If $X$ is Noetherian with a quasi-affine diagonal and $S$ is an arbitrary prestack, then
\[ \Hom(S,X) \to \Fun^L_{\otimes}(D(X),D(S)) \]
is fully faithful with essential image given by functors preserving connective complexes and almost-perfect complexes.
\end{theorem}

Our strategy for proving \autoref{thm:td_Noetherian} is to handle the case of quotient stacks directly, and then reduce to this case using suitable stratifications. The next lemma accomplishes the first of these:

\begin{lemma} \label{lem:global_quotient}
	Let $X = Z/\op{GL}_N$, where $Z$ is a quasiprojective scheme over an affine scheme and the action of $\op{GL}_N$ is linearizable.\footnote{By which we mean there is a $\GL_N$-equivariant embedding $Z \hookrightarrow \mathbb{P}^{N-1}_A$.} Then $X$ is tannakian.
\end{lemma}

\begin{proof}
	Because projective space over an affine scheme can be written as a quasi-affine scheme modulo $\mathbb{G}_m$, we can rewrite $X \simeq Z' / G$ where $Z'$ is quasi-affine and $G = \mathbb{G}_m \times \GL_N$, and thus by \autoref{lem:general_properties} it suffices to prove the claim for $X = BG$. By Lurie's tannaka duality theorem \cite[3.4.2]{DAGVIII} once we assume $F : D(X) \to D(S)$ preserves connective objects, it suffices to show that $F$ preserves flat objects. Any $M \in \QCoh(BG)$ is a union of its coherent subsheaves, and if $M$ is flat then these torsion free coherent sheaves must be locally free, as the same is true for coherent sheaves over on the fppf cover $\Spec(\Z) \to B G$. $F$ preserves locally free sheaves, because they are precisely the dualizable objects of the symmetric monoidal $\infty$-categories $D(X)^{cn}$ and $D(S)^{cn}$. Thus $F(M)$ is a filtered colimit of locally free sheaves, and thus flat.
\end{proof}

To pass from quotient stacks to suitable strata on the target stack $X$ in \autoref{thm:td_Noetherian}, we need:

\begin{lemma} \label{lem:finite_etale}
If $f : Y \to X$ is a finite fppf morphism and $Y$ is tannakian, then $X$ is tannakian.
\end{lemma}
\begin{proof}
By a descent argument as in the proof of \autoref{prop:covering}, it suffices to show that for an affine scheme $S$, any potentially geometric symmetric monoidal functor $F : D(X) \to D(S)$ preserves finite fppf algebras. An algebra $\calA \in D^{\leq 0}(X)$ is finite and fppf if and only if the underlying object of $D^{\leq 0}(X)$ is a locally free sheaf. Because $F$ preserves idempotents in $H_0(\Gamma(\calO_X))$, we can split $X = \bigsqcup_n X_n$ and $S = \bigsqcup S_n$ into disjoint open substacks such that $\calA$ is locally free of rank $n$ on $X_n$ and the functor $D(X) \to D(S_n)$ factors through $D(X_n)$. It suffices to show that the induced functors $D(X_n) \to D(S_n)$ are geometric. On each $X_n$ the locally free sheaf underlying $\calA$ is classified by a map $X_n \to B\GL_n$, and thus a symmetric monoidal functor $D(B\GL_n) \to D(X_n)$ preserving connective and almost perfect complexes. By \autoref{lem:global_quotient}, the composition $D(B\GL_n) \to D(S_n)$ is induced from a map $S_n \to B\GL_n$, which by construction classifies the object of $D^{\leq 0}(S_n)$ underlying $F(\calA)|_{S_n}$.
\end{proof}

The next lemma shows that the description of quasi-affine morphisms given in \S \ref{sec:QuasiAffine} can be modified to take finite presentation constraints into account:

\begin{lemma} \label{lem:lafp_quasiaffine}
Let $X$ be a Noetherian stack, and $U \to X$ be a locally almost finitely presented quasi-affine morphism with $U$ classical. Then $\calO_U$ is a compact localization of an almost finitely presented $\calA \in \CAlg(\QCoh(X))$.
\end{lemma}

\begin{remark}
\autoref{thm:QAchar} implies that if $\calA' \to \calA$ is a map of connective algebras and $\calA \to \calO_U$ and the composition $\calA' \to \calO_U$ are both bounded above compact localizations, then the canonical map $\calA \otimes_{\calA'} \calO_U \to \calO_U$ is an equivalence of $\calA$-algebras. In particular $(\tau^{\leq 0} \calO_U) \otimes_{\calA} \calO_U \to \calO_U$ is an equivalence of $\tau^{\leq 0}\calO_U$-algebras in the previous lemma.
\end{remark}

\begin{proof}
We may assume to start that $\calO_U$ is a compact localization of some discrete algebra $\calB \to \calO_U$, which can be taken to be the algebra $\tau^{\leq 0} \calO_U$. Write $\calB \in \CAlg(\QCoh(X))$ as a filtered union of discrete almost finitely presented algebras $\calB \simeq \bigcup_\alpha \calA_\alpha$. We claim that $\calA_\alpha \to \calO_U$ is a compact localization for some $\alpha$. This claim is local on $X$, so we may assume that $X$ is affine. Choose $f_1,..,f_n \in \calB$ such that $\Spec(\calB_{f_i}) \subset U$ and $\cup_i \Spec(\calB_{f_i}) = U$. By refining our system, we may assume that the $f_i$'s comes from each $\calA_{\alpha}$. Thus, for each $i$, we obtain cartesian squares
\begin{equation} \label{eqn:filtered_union}
\xymatrix{U \ar[r] & \Spec(\calB) \ar[r] & \Spec( \calA_\alpha) \\ U_{f_i} \ar[r]^-{\simeq} \ar[u] & \Spec(\calB[f_i^{-1}]) \ar[r] \ar[u] & \Spec( \calA_{\alpha}[f_i^{-1}]) \ar[u]}
\end{equation}
where the vertical arrows are open immersions. Note that $\calB[f_i^{-1}] = \cup_\alpha \calA_\alpha[f_i^{-1}]$ is an increasing union of the displayed finitely presented algebras. As $U_{f_i}$ is finitely presented, it follows that $\calB[f_i^{-1}] = \calA_\alpha[f_i^{-1}]$ for $\alpha \gg 0$. As there are only finitely many values of $i$ under consideration, we obtain that $\Spec(\calB) \to \Spec(\calA_\alpha)$ is an isomorphism after inverting any $f_i$ for $\alpha \gg 0$. For any such $\alpha$, it follows that $U \to \Spec(\calA_\alpha)$ is a quasi-compact open immersion, and thus $\calA_\alpha \to \calO_U$ is a compact localization.
\end{proof}

The next two lemmas will be useful in showing preservation of finitely presented quasi-affine algebras under the functors appearing in \autoref{thm:td_Noetherian}:

\begin{lemma}
	\label{lem:KernelLocalization}
	Let $R$ be a connective $E_\infty$-ring. Let $A$ be a eventually connective localization of $R$, and let $I \subset \pi_0(R)$ be a finitely generated ideal. If $\pi_0(R)/I \otimes A \simeq 0$, then $M \otimes A \simeq 0$ for any $M \in D_I(R)$.
	\end{lemma}
\begin{proof}
	Let $\calC \subset D(R)$ be the collection of $M$ such that $M \otimes A \simeq 0$. Then $\calC$ is a full stable cocomplete subcategory of $D(R)$. Since $\pi_0(R)/I \in \calC$, by using free resolutions, the essential image of the pushforward functor $D(\pi_0(R)/I) \to D(R)$ belongs to $\calC$. Using triangles, the same is true for the essential image of $D(\pi_0(R)/I^k) \to D(R)$ for all $k \geq 0$. Now, if $M \in D_I(R)$ is discrete, then $M = \cup_k M_k$, where $M_k \subset M$ is the collection of elements annihilated by $I^k$, so it follows that $M \in \calC$. By stability, any bounded object in $D_I(R)$ lies in $\calC$. Finally, since $A$ is eventually connective, it is easy to see that $\calC$ is closed under taking limits of Postnikov towers, i.e., if $\{M_i\}$ is a tower of objects in $\calC$ such that the connectivity of the fibre of  $M_{n+1} \to M_n$ increases with $n$, then $\lim M_i \in \calC$. Putting everything together, we get $D_I(R) \subset \calC$.
\end{proof}

\begin{lemma} \label{lem:preserve_localizations}
Let $X$ be a Noetherian stack, let $R$ be a connective $E_\infty$-algebra, and let $F : D(X) \to D(R)$ be a symmetric monoidal functor preserving connective and almost perfect objects. Then $F$ preserves bounded above compact localizations of connective almost finitely presented algebras.
\end{lemma}

\begin{proof}

Let $\calA \in \CAlg(D^{\leq 0}(X))$ be an almost finitely presented algebra, and let $Z \subset X' := \underline{\Spec}_X(\calA)$ be a (classically) finitely presented closed immersion. Then $F(Z) := \Spec(F(\calO_Z)) \hookrightarrow F(X') := \Spec(F(\calA))$ is a classically finitely presented closed immersion as well. In order to establish the lemma in this case, it suffices to show that $F(\Gamma_Z \calA) \to F(\calA)$ is isomorphic to $\Gamma_{F(Z)} F(\calA) \to F(\calA)$. Consider the induced pullback functor $F : D(Z) \simeq D(X',\calO_Z) \to D(F(X'), F(\calO_Z)) \simeq D(F(Z))$. By construction, this functor commutes with the forgetful functors $i_\ast : D(Z) \to D(X')$ and $i_\ast : D(F(Z)) \to D(F(X'))$. Because $X'$ is Noetherian, the essential image $i_\ast D^b(Z)$ generates $D_Z^b(X')$, and thus $F(\Gamma_Z \calA) \in D_{F(Z)}(F(X'))$. Let $\calC \subset D(F(X'))$ be the full subcategory of objects $M$ such that the canonical map
\[F(\Gamma_Z \calA) \otimes M \to M\]
is an equivalence. Then $\calC \subset D_{F(Z)}(F(X'))$, and contains the image of $F : D(Z) \to D(F(Z))$. \autoref{lem:KernelLocalization} then shows that $D_{F(Z)} (F(X')) \subset \calC$, so the two are equal. The universal property of $\Gamma_{F(Z)} F(\calA) \to F(\calA)$ then identifies this map with $F(\Gamma_Z \calA) \to F(\calA)$.
\end{proof}

\begin{remark}
In fact, using Noetherian approximation on the postnikov tower of an arbitrary connective algebra, one can show that $F$ preserves bounded above compact localizations of any $\calA \in \CAlg(D^{\leq 0}(X))$.

\end{remark}

Combining \autoref{lem:lafp_quasiaffine} and \autoref{lem:preserve_localizations} yields the following:

\begin{corollary} \label{cor:preserve_quasi_affine}
If $F : D(X) \to D(S)$ is a cocontinuous symmetric monoidal functor preserving $\APerf^{cn}$, and $U \to X$ is a locally almost finitely presented quasi-affine map for which $U$ is classical, then $F(\calO_U)$ is a quasi-affine algebra corresponding to a locally almost finitely presented map $F(U) \to S$
\end{corollary}

\begin{remark}
It is possible to show that \autoref{lem:lafp_quasiaffine} holds for arbitrary $n$-truncated $U$, and \autoref{cor:preserve_quasi_affine} holds for arbitrary locally almost finitely presented quasi-affine maps $U \to X$. We have omitted these arguments in the interest of space, because we will only need to consider classical $U$ in the proof of \autoref{thm:td_Noetherian}.
\end{remark}

Next, we record a criterion for smoothness; this will be useful in showing that the functors appearing in \autoref{thm:td_Noetherian} preserve smoothness.

\begin{lemma}
\label{lem:SmoothMapCharacterization}
Let $f:X \to Y$ be an almost finitely presented map of classical schemes. Assume:
\begin{enumerate}
	\item $\Omega^1_{X/Y}$ is locally free, and that each fibre of $f$ is smooth.
	\item For each $x \in X$, one has $\Tor^{\calO_{Y,f(x)}}_1(\calO_{X,x},\kappa(f(x))) = 0$.
\end{enumerate}
Then $f$ is smooth.
\end{lemma}

\begin{proof}
The claim is Zariski-local on $X$ and $Y$, so we may assume that $Y = \Spec(A)$ and $X = \Spec(B)$; assumption (2) then translates to $\Tor_1^A(B,k) = 0$ for any field $k$ under $A$. Choose a factorisation $A \to P \to B$ with $P$ smooth over $A$ (for example, a polynomial algebra), and $P \to B$ surjective with kernel $I$. This leads to a right exact sequence 
\[I/I^2 \to \Omega^1_{P/A} \otimes_P B \stackrel{f^*}{\to} \Omega^1_{B/A} \to 0.\]
Let $K = \ker(f^*)$. Then $K$ is locally free by the assumption on $\Omega^1_{B/A}$. It follows that we can write $I/I^2 = K \oplus K'$, where $K'$ is the kernel of the first map in the sequence. Note that the formation of $K$ commutes with passage to the fibres over $Y$. Moreover, assumption (2) shows that the formation of $I$, and hence $I/I^2$, also commutes with passage to the fibres: for each residue field $k$ of $A$, the kernel of the surjective map $I \otimes_A k \to \ker(P \otimes_A k \to B \otimes_A k)$ is $\Tor_1^A(B,k)$, which vanishes by (2). Also, on each fibre, the map $I/I^2 \to K$ is an isomorphism by smoothness. It follows that $K'/ \mathfrak{m} K' = 0$ for each maximal ideal $\mathfrak{m} \subset A$. Also, since $I$ is finite type (by assumption on $B/A$), so is $K'$ (being a retract of $I/I^2$), and thus Nakayama gives $K' = 0$. This implies that the naive cotangent complex \cite[Tag 00S0]{StacksProject} $I/I^2 \to \Omega^1_{P/A} \otimes_P B$ is a locally free module placed in degree 0, and thus $B/A$ is smooth \cite[Tag 00T2]{StacksProject}.
\end{proof}

\begin{remark}
	Assumption (2) in \autoref{lem:SmoothMapCharacterization} is necessary: the map $\Spec(k) \to \Spec(k[x]/(x^2))$ satisfies assumption (1), but is not smooth.
\end{remark}

To apply \autoref{lem:SmoothMapCharacterization}, we need the following preservation property for (classical) Kahler differentials:

\begin{lemma} \label{lem:preserve_smooth}
	Let $X$ be a classical stack and let $S$ be a classical affine scheme. Let $F : D(X) \to D(S)$ be a cocontinuous symmetric monoidal functor preserving connective objects and quasi-affine localizations. If $U \to X$ is a quasi-affine map with $\Omega^1_{U/X}$ locally free, then $\Omega^1_{F(U)^{cl}/S}$ is also locally free.
\end{lemma}

\begin{proof}
	By hypothesis, $F(\calO_U) \in \CAlg(D(S))$ is a quasi-affine algebra corresponding to a quasi-affine map $F(U) \to S$, and under the equivalence $D(S,F(\calO_U)) \simeq D(F(U))$ we get a symmetric monoidal functor $F' : D(U) \to D(F(U))$ lifting $F$. If $\calA \to \calO_U$ is a compact localization of a connective algebra, $\calA$, then every object in $D(U)^{cn}$ is the restriction of an object of $D(X,\calA)^{cn}$. The same is true for $F(\calA) \to F(\calO_U)$, so $F'$ preserves connective objects as well. The pullback functor for the inclusion $i_{cl} : F(U)^{cl} \hookrightarrow F(U)$ preserves connective objects as well. Thus we have a well-defined symmetric monoidal functor $G := H_0 i_{cl}^\ast F' : \QCoh(U) \to \QCoh(F(U)^{cl})$; as any symmetric monoidal functor preserves dualizable objects, we conclude that $G$ preserves locally free sheaves. It remains to see that $G(\Omega^1_{U/X}) \simeq \Omega^1_{F(U)^{cl}/S}$. For this, by Lurie's cotangent complex formalism, note that $L_{U/X}$ corresponds to $L_{\calO_U/\calO_X}$ under the equivalence $D(U) \simeq D(X,\calO_U)$. Applying this to $F(U) \to S$ as well shows that $F'(L_{U/X}) \simeq L_{F(U)/X}$; here we use that $F$ preserves cotangent complexes. As $L_{F(U)^{cl}/F(U)}$ is 2-connective, we have a canonical isomorphism $H_0 i_{cl}^\ast L_{F(U)/S} \to H_0(L_{F(U)^{cl}/S})$; putting everything together gives $G(\Omega^1_{U/X}) \simeq \Omega^1_{F(U)^{cl}/S}$, as wanted.
\end{proof}

Finally, we can prove \autoref{thm:td_Noetherian}:

\begin{proof} [Proof of \autoref{thm:td_Noetherian}]

	By \autoref{prop:nilpotent_thickening} it suffices to prove the claim when $X$ is classical, and it suffices to show that for any classical $R$, any symmetric monoidal functor $F : D(X) \to D(\Spec(R))$ preserving connective objects and almost perfect objects is geometric. Given a smooth fppf cover by an affine scheme $U \to X$ of relative dimension $d$, \autoref{cor:preserve_quasi_affine} implies that $F(\calO_U) \simeq \calO_{F(U)}$ for some locally almost finitely presented quasi-affine map $F(U) \to \Spec(R)$. We will check that this map is a smooth cover in two steps.
	
	Assume first that the theorem has been proven when $R$ is a field. Then, for general $R$, we will check that $F(U)^{cl} \to \Spec(R)$ is a smooth cover; this is enough by \autoref{prop:covering}. To check this, we verify hypothesis (1) and (2) from \autoref{lem:SmoothMapCharacterization}. For (1), note that we already know that the (derived) fibres of $F(U) \to \Spec(R)$ are smooth and non-empty (by assumption), so the (classical) fibres of $F(U)^{cl} \to \Spec(R)$ are also smooth and non-empty. \autoref{lem:preserve_smooth} then shows that $\Omega^1_{F(U)^{cl}/\Spec(R)}$ is locally free, so (1) is satisfied. For (2), note that $\calO_{F(U)} \to \calO_{F(U)^{cl}}$ has a connected fibre in $D(F(U))$. Hence, for each field $k$ under $R$, the map $\calO_{F(U)} \otimes_R k \to \calO_{F(U)^{cl}} \otimes_R k$ also has a connected fibre. Since the (derived) fibres of $F(U) \to \Spec(R)$ are smooth, the term $\calO_{F(U)} \otimes_R k$ is in the heart of the $t$-structure, and hence the term $\calO_{F(U)^{cl}} \otimes_R k$ has a vanishing $H^{-1}$; this verifies hypothesis (2) from \autoref{lem:SmoothMapCharacterization}, and hence proves the theorem in this case.

	To finish proving the theorem, we may now assume that $R = K$ is a field.  Now any Noetherian stack with affine stabilizer groups, and in particular any quasi-geometric Noetherian stack, admits a finite stratification by reduced locally closed substacks $X_i$ such that for each $X_i$ there is a finite fppf morphism $Y_i \to X_i$ such that $Y_i$ is the quotient of a scheme which is quasiprojective over a Noetherian affine scheme by a linear action of $\GL_n$. This is \cite[Proposition 2.3.4]{drinfeld2013some} -- the result there is stated only for stacks defined over a field, but the proof applies verbatim for stacks over $\Z$ as well. Applying \autoref{lem:finite_etale} and \autoref{lem:global_quotient} inductively, it will therefore suffice to show that if $X$ is a quasi-geometric stack with an open substack $U \subset X$ such that both $U$ and $Z = X \setminus U$ are tannakian, then $X$ is tannakian. There are two cases:

\medskip
\noindent \textit{Case $F(\calO_Z)=0$:}
\medskip

In this case  $F(M) = 0$ for any complex of $\calO_Z$-modules (regarded as $\calO_X$-modules), because any such $M$ is a retract of $M\otimes_{\calO_X} \calO_Z$. The category $\Coh_Z(X)$ is generated by $\calO_Z$-modules, and $\QCoh_Z(X)$ is generated by $\Coh_Z(X)$, so $F(M) = 0$ for any $M \in D(X)$ with $H_i(M) \in \QCoh_Z(X)$ and $H_i(M) = 0$ for $i\ll 0$. In particular $F(R\Gamma_Z \calO_X) = 0$, so applying $F$ to the exact triangle $R\Gamma_Z \calO_X \to \calO_X \to \calO_U$ shows that $K \to F(\calO_U)$ is an isomorphism. This implies that $F$ factors uniquely through $D(X) \to D(U)$, and hence there is a map $\Spec(K) \to U \to X$ inducing $F$.

\medskip
\noindent \textit{Case $F(\calO_Z) \neq 0$:}
\medskip

Here $K = H_0 F(\calO_X) \to H_0 F(\calO_Z)$ is surjective, and $F(\calO_Z)$ is a ring, so $H_0 F(\calO_Z) \simeq K$. This defines a map of $K$-algebras $F(\calO_Z) \to K$. Consider the commutative square
$$\xymatrix{D(X) \ar[d]^F \ar[rr]^-{(\bullet)\otimes \calO_Z} & & D(X,\calO_Z) \ar[d]^{F_Z} & \\ D(\Spec(K)) \ar[rr]^-{(\bullet)\otimes F(\calO_Z)} & & D(\Spec(K),F(\calO_Z)) \ar[r] & D(\Spec(K)) }.$$
The composition of the lower horizontal maps is canonically isomorphic to the identity functor, so $F$ factors through the pullback map $D(X) \to D(Z)$. Therefore by hypothesis $F$ is induced by a morphism $\Spec(K) \to Z \to X$.

\end{proof}

\section{Strange pushouts of schemes}
\label{sec:strange_pushout}

Let $X$ be a smooth variety over a field with two smooth closed subvarieties $Z_0, Z_1 \subset X$ meeting transversally. Let $X_i = \op{Bl}_{Z_i} X$ and let $\tilde{Z}_i \subset X_{1-i}$ be the strict transform. We recall some facts from \cite{li2009wonderful}: We have a canonical identification $X_{01} := \op{Bl}_{\tilde{Z_0}} X_1 \simeq \op{Bl}_{\tilde{Z_1}} X_0$.
\begin{equation} \label{eqn:strange_pushout}
\xymatrix{ & & X_{01} \ar[dl] \ar[dr] & & \\ \tilde{Z_1} \ar@{^{(}->}[r] & X_0 \ar[dr]^{\pi_0} & & X_1 \ar[dl]_{\pi_1} & \tilde{Z}_0 \ar@{_{(}->}[l] \\ & Z_0 \ar@{^{(}->}[r] & X & Z_1 \ar@{_{(}->}[l] & }
\end{equation}
Also, $\tilde{Z}_i \simeq \op{Bl}_Z Z_i$, where $Z := Z_0 \cap Z_1$. We have two smooth divisors $D_0,D_1 \subset X_{01}$ which are the \emph{dominant transforms} of $Z_0$ and $Z_1$ \cite{li2009wonderful}. $D_0$ can be described equivalently as the strict transform in $X_{01}$ of the preimage of $Z_0$ under $X_0 \to X$, or as the preimage of $\tilde{Z}_0$ under $X_{01} \to X_1$, and $D_1$ has a symmetric description. The divisors $D_0$ and $D_1$ meet transversally.

\begin{proposition} \label{prop:strange_pushout}
The diagram \eqref{eqn:strange_pushout} is a pushout in the category of tannakian stacks.
\end{proposition}

We collect some preliminary lemmas before proving this proposition.

\begin{lemma} \label{lem:tor_independence}
Let $X$ be a smooth variety and let $V, Z \subset X$ be smooth subvarieties meeting transversally. Let $\tilde{V} \subset \Bl_Z X$ be the strict transform of $V$. Then the diagram
$$\xymatrix{\tilde{V} \ar[r] \ar[d] & \Bl_Z X \ar[d]^\pi \\ V \ar[r] & X }$$
is a pullback diagram, and it is Tor independent.
\end{lemma}
\begin{proof}
The claim is local on $X$, so let us assume that $X$ is affine and $V$ is defined by equations $f_1,\ldots,f_n$ such that $[df_i]_p : T_p X \to k^n$ is surjective at each $p \in V$. The fact that $V$ meets $Z$ transversally means that $T_p Z + \ker [df_i]_p = T_p X$ for all $p \in Z \cap Y$.

The scheme theoretic preimage $\pi^{-1}(V)$ is defined by the pullback of these equations to $\Bl_Z X$, and if $q \in \Bl_Z X$ is a point with $\pi(q) = p \in V$, then the Jacobian matrix $[df_i|_{\Bl_Z X}]_q$ is the composition $T_q \Bl_Z X \to T_p X$ with $[df_i]_p$. Because $\pi$ is an isomorphism away from $Z$, the Jacobian has full rank at all points of $V \setminus Z$. For $q \in \Bl_Z X$ lying over $p \in Z$, the image of $T_q \Bl_Z X \to T_p X$ contains $T_p Z$, and it follows from the fact that $T_p Z + \ker[df_i]_p = T_p X$ that the Jacobian $[df_i|_{\Bl_Z X}]_q : T_q \Bl_Z X \to k^n$ is surjective.

We have shown that the preimage $\pi^{-1} V$ is smooth, and it agrees with $V$ in the complement of the exceptional divisor of $\Bl_Z X$. Therefore it must agree with the strict transform of $V$. We have actually shown more: the functions $f_1,\ldots,f_n$ which form a regular sequence for the structure sheaf $\calO_V$ are still a regular sequence for the structure sheaf of $\pi^{-1} (V)$. Therefore $L\pi^\ast \calO_V \simeq \calO_{\pi^{-1} (V)}$, and hence the square is tor independent.
\end{proof}

We will use the following fact from \cite{bondal2002derived}: If $Y \subset X$ is a smooth subvariety of codimension $c+2$ and $X$ is smooth, then we have a semiorthogonal decomposition
\begin{equation} \label{eqn:main_SOD}
\Perf(\op{Bl}_Y X) = \sod{\Perf(Y),\ldots,\Perf(Y)(c Y'),\Perf(X)}
\end{equation}
where $Y' \to Y$ is the preimage in $\Bl_Y X$ of $Y$. The fully faithful embedding $\Perf(X) \to \Perf(\Bl_Y X)$ is the pullback functor, and for each $n$ the embedding $\Perf(Y) \simeq \Perf(Y)(nY') \subset \Perf(\Bl_Y X)$ is given by the pullback followed by the pushforward along the inclusion, $\Perf(Y) \to \Perf(Y') \to \Perf(\Bl_Y X)$, followed by the tensor product with the invertible sheaf $\calO_{\Bl_Y X}(n Y')$.

\begin{lemma} \label{lem:perf_product}
The canonical functor of symmetric monoidal $\infty$-categories $\Perf(X) \to \Perf(X_0) \times_{\Perf(X_{01})} \Perf(X_1)$ is an equivalence.
\end{lemma}
\begin{proof}
Let $Z_1' := \pi_1^{-1}(Z_1) \subset X_1$, then as $D_1$ is the strict transform of $Z_1'$, $\calO_{X_{01}}(D_1)$ is the preimage of $\calO_{X_1}(Z_1')$ under the projection $X_{01} \to X_1$. A similar description applies to $D_0$. Therefore if we apply the semiorthogonal decomposition \eqref{eqn:main_SOD} to the iterated blowups defining $X_{01}$ we get
\begin{equation} \label{eqn:semiorthogonal_decomposition}
\begin{array}{rl}
\Perf(X_{01}) &= \sod{\Perf(\tilde{Z}_1),\ldots,\Perf(\tilde{Z}_1)(c_1 D_1),\underbrace{\Perf(Z_0),\ldots,\Perf(Z_0)(c_0 D_0),\Perf(X)}_{\simeq \Perf(X_0)}} \\
&= \sod{\Perf(\tilde{Z}_0),\ldots,\Perf(\tilde{Z}_0)(c_0 D_0),\underbrace{\Perf(Z_1),\ldots,\Perf(Z_1)(c_1 D_1),\Perf(X)}_{\simeq \Perf(X_1)}}
\end{array}
\end{equation}
In this notation the twist $(\bullet) (n D_i)$ corresponds to tensoring a subcategory with the invertible sheaf $\calO_{X_{01}}(n D_i)$. The embedding $\Perf(Z_i) \subset \Perf(X_{01})$ denotes the composition
$$\Perf(Z_i) \xrightarrow{\text{pullback}} \Perf(Z_i') \xrightarrow{\text{pushforward}} \Perf(X_i) \xrightarrow{\text{pullback}} \Perf(X_{01})$$
And the embedding $\Perf(\tilde{Z}_i) \subset \Perf(X_{01})$ is the composition
$$\Perf(\tilde{Z}_i) \xrightarrow{\text{pullback}} \Perf(D_i) \xrightarrow{\text{pushforward}} \Perf(X_{01})$$

Because all of the pullback functors from $\Perf(X_0),\Perf(X_1),$ and $\Perf(X)$ to $\Perf(X_{01})$ are fully faithful, we can think of all three categories as subcategories of $\Perf(X_{01})$, and we must simply show that $\Perf(X) = \Perf(X_0) \cap \Perf(X_1)$. From the above semiorthogonal decompositions, this is equivalent to showing that the subcategories $\Perf(\tilde{Z}_1),\ldots,\Perf(\tilde{Z}_1)(c_1 D_1),\Perf(\tilde{Z}_0),\ldots,\Perf(\tilde{Z}_0)(c_0 D_0)$ generate the right orthogonal of $\Perf(X)$.

Note that $Z'_0$ and $\tilde{Z}_1$ intersect transversally in $X_0$. By \autoref{lem:tor_independence}, the square
$$\xymatrix{D_0 \ar[d] \ar[r] & X_{01} \ar[d] \\ Z_0' \ar[r] & X_0}$$
is a pullback square, and furthermore it is Tor independent. Therefore pushforward along the horizontal arrows commutes with pullback along the vertical arrows. Hence the embedding of categories $\Perf(Z_i) \to \Perf(X_{01})$ described above can actually be identified with the composition
$$\Perf(Z_i) \xrightarrow{\text{pullback}} \Perf(D_i) \xrightarrow{\text{pushforward}} \Perf(X_{01}).$$

The claim follows from this: we already know the right orthogonal of $\Perf(X)$ in $\Perf(X_{01})$ is generated by $\Perf(\tilde{Z}_1),\ldots,\Perf(\tilde{Z}_1)(c_1 D_1)$ and the subcategories $\Perf(Z_0),\ldots,\Perf(Z_0)(c_0 D_0)$. However, we have just shown that the embedding $\Perf(Z_0) \to \Perf(X)$ factors through the pullback functor $\Perf(Z_0) \to \Perf(\tilde{Z}_0)$, so the subcategories $\Perf(\tilde{Z}_0),\ldots \Perf(\tilde{Z}_0)(c_0 D_0)$ will also suffice to generate the right orthogonal of $\Perf(X)$.
\end{proof}

With a little more care, one can strengthen \autoref{lem:perf_product}.

\begin{corollary}
The conclusion of \autoref{lem:perf_product} is true with $\Perf$ replaced by $\APerf$, $D(\bullet)$, $\APerf^{cn}$, or $D(\bullet)^{cn}$.
\end{corollary}
\begin{proof}
\textit{The case of $D(\bullet)$:}
\medskip

$X_0,X_1,$ and $X_{01}$ are perfect stacks, so for purely formal reasons the semiorthogonal decomposition of \eqref{eqn:semiorthogonal_decomposition} extends to the category $D(X_{01})$, replacing $\Perf$ with $D(\bullet)$ everywhere. The same argument as above shows that $D(X) = D(X_0) \cap D(X_1)$, the intersection being taken in $D(X_{01})$, and thus the canonical map $D(X) \to D(X_0)\times_{D(X_{01})} D(X_1)$ is an equivalence.

\medskip
\noindent \textit{The case of $\APerf$ and $\APerf^{cn}$:}
\medskip

An object in $\APerf(X_0) \times_{\APerf(X_{01})} \APerf(X_1)$ is the pullback of a uniquely determined object of $D(X)$. However, if $F \in D(X)$ and $\pi^\ast_0 F \in \APerf(X_0)$, then $F \in \APerf(X)$ because $F \simeq (\pi_0)_\ast \pi_0^\ast F$ and $(\pi_0)_\ast$ preserves almost perfect complexes. The claim for connective almost perfect objects follows from the fact that an almost perfect object is connective if and only if its restriction to every closed point is connective, by Nakayama's lemma. Thus if $\pi_0^\ast F$ is connective and $F$ is almost perfect, then $F$ is connective.

\medskip
\noindent \textit{The case of $D(X)^{cn}$:}
\medskip

We claim that for any $F \in D(X)$, $F$ is connective if and only if $\pi_0^\ast F$ is connective. Note that because $\pi_0$ has finite cohomological dimension and $F \simeq (\pi_0)_\ast \pi_0^\ast F$, we have that $F$ is eventually connective and thus has a lowest (in homological degree) nonvanishing homology sheaf. It will suffice to show that the pullback functor $H_0 (\pi_0)^\ast : \QCoh(X) \to \QCoh(X_0)$ is conservative, so that the lowest nonvanishing homology sheaf of $\pi^\ast F$ will occur in the same degree as the lowest nonvanishing homology sheaf of $F$.

To show that the classical pullback is conservative, let $M \in \QCoh(X)$, and write $M = \varinjlim F_\alpha$ with $F_\alpha \in \Coh(X)$. Then $H_0 \pi_0^\ast M = \varinjlim H_0 \pi_0^\ast F_\alpha$ is zero if and only if for all $\alpha$, there is a $\beta$ larger than $\alpha$ such that the map $H_0 \pi_0^\ast F_\alpha \to H_0 \pi_0^\ast F_\beta$ is $0$. But for coherent sheaves, if $H_0 \pi_0^\ast$ of a morphism vanishes, then that morphism must have already been $0$.
\end{proof}

\begin{proof}[Proof of \autoref{prop:strange_pushout}]
According to \autoref{defn:tannakian}, we can identify $\Map(X,Y)$ with the full subcategory of $\Fun^L_\otimes(D(Y)^{cn},D(X)^{cn})$ preserving almost perfect complexes. By the previous corollary we can identify this with \[\Fun^L_\otimes(D(Y)^{cn},D(X_1)^{cn}) \times_{\Fun^L_\otimes(D(Y)^{cn},D(X_{01})^{cn})} \Fun^L_\otimes(D(Y)^{cn},D(X_0)^{cn}), \]
and a under this identification a functor preserves almost perfect complexes if and only if its restriction to $D(X_i)^{cn}$ preserves almost perfect complexes, for $i=0,1,01$.
\end{proof}

\section{Mapping stacks}
\label{sect:mapping_stacks}

One of the original applications of Lurie's tannakian formalism in \cite{DAGXIV} was to the algebraicity of mapping stack functor $\inner{\Map}_S(X,Y)$ when $X \to S$ is a flat, strongly proper, locally finitely presented map of spectral Deligne-Mumford stacks and $Y \to S$ is a locally finitely presented geometric $S$-stack. In \cite{halpern2014mapping}, the second author and A. Preygel develop Lurie's method further to show that $\inner{\Map}_S(X,Y)$ is algebraic for a much larger class of morphisms $X \to S$ which are flat and ``cohomologically proper.'' One key step is establishing that the mapping stack functor is integrable, meaning for any discrete Noetherian ring $R$ over $S$ which is complete with respect to an ideal $I \subset R$, the canonical map
\begin{equation}
\Map_S(X \times_S \Spec(R),Y) \to \Map_S(X \times_S \op{Spf}(R),Y)
\end{equation}
is an equivalence, where $\op{Spf}(R) = \colim \Spec(R/I^n)$ is the formal completion along $I$.

Using Lurie's tannaka duality theorem, one can prove that the mapping stack is integrable provided $Y$ is locally finitely presented over $S$ and has an affine diagonal. \autoref{thm:td_Noetherian} allows one to extend this result to $Y$ which have quasi-affine diagonal, and significantly simplifies the proof of integrability. In addition, this version of the tannaka duality theorem does not require any noetherian or finite presentation hypotheses on $X$.

\begin{corollary}
Let $X \to S \leftarrow Y$ be stacks with $Y$ finitely presented and algebraic over $S$ with quasi-affine diagonal, and let $R$ be a complete Noetherian algebra over $S$. If $X$ is such that $\APerf(X \times_S \Spec(R))^{cn} \to \APerf(X \times_S \op{Spf}(R))^{cn}$ is an equivalence, then
\[
\Map_S(X \times_S \Spec(R), Y) \to \Map_S(X \times_S \op{Spf}(R),Y)
\]
is an equivalence.
\end{corollary}

\begin{proof}
It by base change it suffices to consider the case $S = \Spec(R)$, so by the tannaka duality of \autoref{thm:td_Noetherian}, we must show that
\[
\Fun_{\otimes,R}^L(D(Y),D(X))' \to \Fun_{\otimes,R}^L(D(Y),D(\widehat{X}))'
\]
is an equivalence, where $\widehat{X} := \op{Spf}(R) \times_S X$ is the derived formal completion, and $\Fun_{\otimes,R}^L(\bullet,\bullet)'$ denotes the $\infty$-category of $R$-linear symmetric monoidal $\infty$-functors which preserve connective and almost perfect complexes. By \autoref{lem:small_cat_Noetherian} it suffices to show that
\[
\Fun_{\otimes,R}^L(\APerf(Y)^{cn},\APerf(X)^{cn}) \to \Fun_{\otimes,R}^L(\APerf(Y)^{cn},\APerf(\widehat{X})^{cn})
\]
is an equivalence. Which follows immediately from the hypotheses.

\end{proof}

\begin{remark}
When $X \times_S \Spec(R)$ is Noetherian, then $\APerf(X) \to \APerf(\widehat{X})$ is an equivalence if and only if $\APerf(\widehat{X})^{cn}$ is an equivalence.
\end{remark}

\begin{remark}
The conclusion of the previous corollary remains valid if instead we assume that $Y$ is only locally finitely presented over $S$ and that $X \to S$ is algebraic and qcqs. In this case, quasi-compactness of $X$ forces any pair of maps $X \to Y$ and $X_0 = X \times_S \Spec(R/I) \to Y$ to factor through some quasi-compact open substack $U \subset Y$, and any map $\widehat{X} \to Y$ factors through this same $U \subset Y$. It thus suffices to show that $\Map_S(X \times_S \Spec(R),Y) \to \Map_S(X \times_S \op{Spf}(R),Y)$ is an equivalence on the connected components corresponding to maps which factor through $U \subset Y$ for every quasi-compact open substack, which reduces to the previous corollary.
\end{remark}

\subsection{A non-derived integrability result}
\label{ss:NonDerived}

In this short section, we record a direct argument for integrability of the mapping stack when the source is noetherian and the target has a mild separation condition. The argument we give is tannakian, but does not require the full strength of the tannaka duality theorem.

For the rest of this section, fix a noetherian ring $R$ which is $I$-adically complete for some ideal $I$. For any $R$-stack $Z$, write $Z_n$ for its reduction modulo $I^n$. Then we have:

\begin{theorem}
Let $X \stackrel{f}{\to} \Spec(R) \stackrel{g}{\gets} Y$ be finitely presented maps of stacks. Assume that $f$ is proper, and that $g$ has an affine diagonal. Then the canonical map gives an equivalence $$\Map_R(X,Y) \simeq \lim \Map_R(X_n,Y).$$
\end{theorem}

Recall that compact objects in $\QCoh(Y)$ are exactly coherent sheaves, and that $\QCoh(Y)$ is generated under colimits by $\Coh(Y)$. We must check that a compatible system $\{f_n:X_n \to Y\}$ of maps algebraizes to a map $f:X \to Y$. Let $\{F_n:\QCoh(Y) \to \QCoh(X_n)\}$ be the corresponding system of pullback functors. We obtain an induced system $\{g_n:\Coh(Y) \to \Coh(X_n)\}$ by restricting to compact objects. As $\Coh(Y) \simeq \lim \Coh(X_n)$, there is an induced right-exact symmetric monoidal functor $G:\Coh(Y) \to \Coh(X)$. Passing to ind-completions, we obtain a cocontinuous symmetric monoidal functor $F:\QCoh(Y) \to \QCoh(X)$ preserving coherent sheaves.  The above theorem then follows from:

\begin{proposition}
The functor $F$ is geometric, i.e., $F = f^*$ for some map $f:X \to Y$ of $R$-stacks.
\end{proposition}

\begin{proof}
Let $g:U \to Y$ be a smooth cover by an affine scheme $U$, and let $\calA := g_* \calO_U \in \QCoh(Y)$ be the corresponding commutative algebra. This algebra is finitely presented as $g$ is so. Let $p:V \to X$ be the affine morphism defined by the algebra $F(\calA) \in \QCoh(X)$. As $F$ preserves finitely presented sheaves, one easily checks that $F$ also preserves finitely presented algebras, so $p$ is finitely presented. By smooth descent, it is enough to show that the smooth locus $V^\circ \subset V$ of $p$ covers $X$, i.e., that $V^\circ \to X$ surjective. Note that $V^\circ \to X$ is a smooth map, so its image is an open subset of $X$. As $X$ is proper over a complete ring $R$, the only open set containing $X_0$ is all of $X$, so it is enough to show that $p:V \to X$ is smooth at all points over $X_0$. By assumption, we know that the induced map $p_n:V_n \to X_n$ is a smooth cover, so $V \to X$ is formally smooth at all points lying over $X_0$. Grothendieck's theorem on formal smoothness then implies that $V \to X$ is formally smooth at all points of $V$ lying above $X_0$, as desired.
\end{proof}

\section{Algebraization of formal points}

\begin{theorem}
	Let $A$ be a commutative ring that is $I$-adically complete for some ideal $I$. Let $X$ be either a noetherian quasi-geometric algebraic stack, or a classical quasi-compact separated tame Deligne-Mumford stack. Then the natural map gives an equivalence
	\[ X(A) \simeq \lim X(A/I^n).\]
\end{theorem}
\begin{proof}
We first note that the hypothesis on $X$ is stable under passage to the underlying classical stack. Moreover, since both $X(A)$ and $X(A/I^n)$ depend only on the underlying classical stack (as $A$ is discrete), we may assume $A$ is classical.

Assume now that $X$ is noetherian with quasi-affine diagonal. Then, by \autoref{thm:td_Noetherian} and \autoref{lem:small_cat_Noetherian}, we have $X(A) \simeq \Fun^L_{\otimes}(\APerf(X)^{cn},\APerf(A)^{cn})$, and similarly for $A/I^n$-points. Hence, it suffices to show that $\APerf(A)^{cn} \simeq \lim \APerf(A/I^n)^{cn}$ in the $\infty$-category of symmetric monoidal $\infty$-categories with finitely cocontinuous functors. This follows from \autoref{lem:PerfAPerfContinuity} together with the following observation: $K \in \APerf(A)$ is connective if and only if $K \otimes_A A/I \in \APerf(A/I)$ is connective. For the latter statement, one direction clear; conversely, if $K \otimes_A A/I$ is connective, so is $K \otimes_A A/I^n$ for all $n$. Since $K = \lim K \otimes_A A/I^n$, it suffices to note that the projective system $\{H^0(K \otimes_A A/I^n)\}$ has surjective transition maps.

Assume now that $X$ is a classical quasi-compact separated tame Deligne-Mumford stack. Then, by \autoref{cor:tame_DM}, we have $X(A) \simeq \Func_{\otimes}(\Perf(X), \Perf(A))$. It thus remains to observe that $\Perf(A) \simeq \lim \Perf(A/I^n)$ in the $\infty$-category of stable symmetric monoidal $\infty$-categories with finitely cocontinuous functors. This is again covered by \autoref{lem:PerfAPerfContinuity}
\end{proof}

\begin{lemma}
\label{lem:PerfAPerfContinuity}
Let $A$ be a commutative ring that is $I$-adically complete for some ideal $I$. Then $\Perf(A) \simeq \lim \Perf(A/I^n)$, and $\APerf(A) \simeq \lim \APerf(A/I^n)$.
\end{lemma}
\begin{proof}
The statement for $\Perf(-)$ can be found in \cite[Lemma 4.2]{bhatt2014algebraization}. For $\APerf(-)$, once the full faithfulness of $\APerf(A) \to \lim \APerf(A/I^n)$ has been established, the essential surjectivity comes from \cite[Tag 09AV]{StacksProject}. For full faithfulness, fix some $K \in \APerf(A)$.  We claim that $K \simeq \lim K \otimes_A A/I^n$ via the natural map. To see this, we can represent $K$ by an explicit bounded above complex $K^\bullet$ whose terms are finite projective. Then $K \otimes_A A/I^n$ is computed by $K^\bullet/I^n$, so the desired claim follows from the $I$-adic completeness of $A$. Using this, one obtains that if $K,L \in \APerf(A)$, then 
\[ \Map_A(K,L) \simeq \lim \Map_A(K, L \otimes_A A/I^n) \simeq \lim \Map_{A/I^n}(K \otimes_A A/I^n, L \otimes_A A/I^n),\]
which proves the full faithfulness.
\end{proof}

\begin{remark}
The same method can be used to prove a closely related result, which will be explained and applied in \cite{halpern2014on}: Let $\Theta := \mathbb{A}^1 /\mathbb{G}_m$ over some base Noetherian ring $k$ and $\mathbb{G}_m$ acts with weight $1$, and let $\Theta_n := \op{Spec}(k[x] / (x^n)) / \mathbb{G}_m$. Then the restriction map
$$\Map_k(\Theta,X) \to \lim_n \Map_k(\Theta_n,X)$$
is an equivalence of $\infty$-categories for any locally Noetherian quasi-geometric $k$-stack $X$.
\end{remark}

\bibliographystyle{alpha}
\bibliography{td-stacks}

\end{document}